\newif\ifshowflag
\showflagtrue   
\showflagfalse   
\newif\ifshowqstn
\showqstntrue   
\showqstnfalse   
\newif\ifshowinfo
\showinfotrue   
\showinfofalse   

\newcommand{\flag}[1]{\ifshowflag
  {\noindent\color{red}{$\clubsuit\clubsuit\clubsuit$\; {\sffamily #1}\; $\clubsuit\clubsuit\clubsuit$}}\fi}

\documentclass[12pt]{amsart}
\usepackage{amssymb,amscd}

\usepackage{ifthen}
\usepackage{tikz}
\usetikzlibrary{positioning,arrows,calc,shapes}
\usetikzlibrary{shapes.multipart,shapes.misc}
\usetikzlibrary{patterns}
\usetikzlibrary{decorations.pathreplacing,decorations.markings}
\usepackage{xcolor}         

\usepackage[margin=10pt,font=small,labelfont=bf, labelsep=endash]{caption}
\usepackage{subcaption}

\usepackage[shortlabels]{enumitem}

\usepackage{braket}

\numberwithin{equation}{section}        

\newcommand{\Va}{V\"ais\"al\"a}     

\catcode`\@=11       

\def\rf#1{\@rf{#1}#1:;;}
\def\rfs#1{\@rfs{#1}#1:;;}
\def\rfm#1{\@rfF#1<>;;}

\def\@C{C}\def\@CC{CC}\def\@E{E}\def\@F{F}\def\@L{L}\def\@P{P}\def\@PP{PP}\def\@Q{Q}
\def\@R{R}\def\@S{S}\def\@T{T}\def\@TT{TT}\def\@X{X}\def\@XX{XX}\def\@Ex{Ex}
\def\@s{s}\def\@ss{ss}\def\@f{f}

\def\@rf#1#2:#3;;{\def\@b{#2}
  \ifx\@b\@C Corollary~\ref{#1}\else%
  \ifx\@b\@CC Corollary~\ref{#1}\else%
  \ifx\@b\@E (\ref{#1})\else
  \ifx\@b\@Ex Exercise~\ref{#1}\else%
  \ifx\@b\@F Fact~\ref{#1}\else%
  \ifx\@b\@L Lemma~\ref{#1}\else%
  \ifx\@b\@P Proposition~\ref{#1}\else%
  \ifx\@b\@PP Proposition~\ref{#1}\else%
  \ifx\@b\@Q Question~\ref{#1}\else%
  \ifx\@b\@R Remark~\ref{#1}\else%
  \ifx\@b\@S Section~\ref{#1}\else%
  \ifx\@b\@T Theorem~\ref{#1}\else%
  \ifx\@b\@TT Theorem~\ref{#1}\else%
  \ifx\@b\@X Example~\ref{#1}\else%
  \ifx\@b\@XX Example~\ref{#1}\else%
  \ifx\@b\@s \S\ref{#1}\else
  \ifx\@b\@ss \S\ref{#1}\else
  \ifx\@b\@f Figure~\ref{#1}\else%
  \ref{#1}\fi\fi\fi\fi\fi\fi\fi\fi\fi\fi\fi\fi\fi\fi\fi\fi\fi\fi}

\def\@rfs#1#2:#3;;{\def\@b{#2}
  \ifx\@b\@C Corollaries~\ref{#1}\else%
  \ifx\@b\@CC Corollaries~\ref{#1}\else%
  \ifx\@b\@Ex Exercises~\ref{#1}\else%
  \ifx\@b\@F Facts~\ref{#1}\else%
  \ifx\@b\@L Lemmas~\ref{#1}\else%
  \ifx\@b\@P Propositions~\ref{#1}\else%
  \ifx\@b\@PP Propositions~\ref{#1}\else%
  \ifx\@b\@Q Questions~\ref{#1}\else%
  \ifx\@b\@R Remarks~\ref{#1}\else%
  \ifx\@b\@S Sections~\ref{#1}\else%
  \ifx\@b\@T Theorems~\ref{#1}\else%
  \ifx\@b\@TT Theorems~\ref{#1}\else%
  \ifx\@b\@X Examples~\ref{#1}\else%
  \ifx\@b\@XX Example~\ref{#1}\else%
  \ifx\@b\@s \S\S\ref{#1}\else
  \ifx\@b\@ss \S\S\ref{#1}\else
  \ifx\@b\@f Figures~\ref{#1}\else%
  \ref{#1}\fi\fi\fi\fi\fi\fi\fi\fi\fi\fi\fi\fi\fi\fi\fi\fi\fi}

\def\@rfF<#1>#2;;{\def\@c{#2}
  \@rfs{#1}#1:;;\ifx\@c\empty\else\@rfL:#2;;\fi}

\def\@rfL:#1<#2>#3;;{\def\@b{#2}\def\@c{#3}
  #1\ifx\@b\empty\else\ref{#2}\ifx\@c\empty\else\@rfL:#3;;\fi\fi}

\catcode`\@=12       

\definecolor{darkblue}{rgb}{0,0,0.6}
\definecolor{darkgreen}{rgb}{0,0.4,0}
\definecolor{darkred}{rgb}{0.6,0,0}
\definecolor{lightblue}{rgb}{0.8,0.8,1}
\definecolor{lightgreen}{rgb}{0.25,1,0.25}
\definecolor{lightred}{rgb}{1,0.5,0.5}
\definecolor{lightpurple}{rgb}{1,0.4,0.6}
\definecolor{darkpurple}{rgb}{0.5,0,0.5}

\newcounter{pos}
\tikzset{
  initcounter/.code={\setcounter{pos}{0}},
  style between/.style n args={3}{
    postaction={
      initcounter,
      decorate,
      decoration={
        show path construction,
        curveto code={
          \addtocounter{pos}{1}
          \pgfmathtruncatemacro{\min}{#1 - 1}
          \ifthenelse{\thepos < #2 \AND \thepos > \min}{
            \draw[#3]
            (\tikzinputsegmentfirst)
            ..
            controls (\tikzinputsegmentsupporta) and (\tikzinputsegmentsupportb)
            ..
            (\tikzinputsegmentlast);
          }{}
        }
      }
    },
  },
}

\newlength{\hatchspread}
\newlength{\hatchthickness}
\newlength{\hatchshift}
\newcommand{\hatchcolor}{}
\tikzset{hatchspread/.code={\setlength{\hatchspread}{#1}},
         hatchthickness/.code={\setlength{\hatchthickness}{#1}},
         hatchshift/.code={\setlength{\hatchshift}{#1}},
         hatchcolor/.code={\renewcommand{\hatchcolor}{#1}}}
\tikzset{hatchspread=3pt,
         hatchthickness=0.4pt,
         hatchshift=0pt,
         hatchcolor=black}
\pgfdeclarepatternformonly[\hatchspread,\hatchthickness,\hatchshift,\hatchcolor]
   {custom north west lines}
   {\pgfqpoint{\dimexpr-2\hatchthickness}{\dimexpr-2\hatchthickness}}
   {\pgfqpoint{\dimexpr\hatchspread+2\hatchthickness}{\dimexpr\hatchspread+2\hatchthickness}}
   {\pgfqpoint{\dimexpr\hatchspread}{\dimexpr\hatchspread}}
   {
    \pgfsetlinewidth{\hatchthickness}
    \pgfpathmoveto{\pgfqpoint{0pt}{\dimexpr\hatchspread+\hatchshift}}
    \pgfpathlineto{\pgfqpoint{\dimexpr\hatchspread+0.15pt+\hatchshift}{-0.15pt}}
    \ifdim \hatchshift > 0pt
      \pgfpathmoveto{\pgfqpoint{0pt}{\hatchshift}}
      \pgfpathlineto{\pgfqpoint{\dimexpr0.15pt+\hatchshift}{-0.15pt}}
    \fi
    \pgfsetstrokecolor{\hatchcolor}
    \pgfusepath{stroke}
   }

\def\centerarc[#1](#2)(#3:#4:#5)
  { \draw[#1] ($(#2)+({#5*cos(#3)},{#5*sin(#3)})$) arc (#3:#4:#5); }

\newcommand{\ds}{\displaystyle} \newcommand{\half}{\frac{1}{2}}       \newcommand{\qtr}{\frac{1}{4}}
                \newcommand{\xra}{\xrightarrow}

\newcommand{\comp}{\circ}           	        
\newcommand{\sm}{\setminus}                     

\newcommand{\lex}{\lesssim} \newcommand{\gex}{\gtrsim}  
\newcommand{\eqx}{\simeq}

\def\vint_#1{\mathchoice
          {\mathop{\vrule width 6pt height 3 pt depth -2.5pt
                  \kern -8pt \intop}\nolimits_{\kern -4pt#1}}%
          {\mathop{\vrule width 5pt height 3 pt depth -2.6pt
                  \kern -6pt \intop}\nolimits_{#1}}%
          {\mathop{\vrule width 5pt height 3 pt depth -2.6pt
                  \kern -6pt \intop}\nolimits_{#1}}%
          {\mathop{\vrule width 5pt height 3 pt depth -2.6pt
                   \kern -6pt \intop}\nolimits_{#1}}}

\newcommand{\ifff}{if and only if }  \newcommand{\wrt}{with respect to }  \newcommand{\param}{parametrization}
    \newcommand{\tfaqe}{the following are quantitatively equivalent}

        \newcommand{\holo}{holomorphic}

       \newcommand{\qc}{quasiconformal}
\newcommand{\qsy}{quasisymmetry}    \newcommand{\qsc}{quasisymmetric}
\newcommand{\qiy}{quasiisometry}    \newcommand{\qic}{quasiisometric}
\newcommand{\homeo}{homeomorphism}  \newcommand{\homic}{homeomorphic}

\newcommand{\Mob}{M\"obius}        \newcommand{\MT}{\Mob\ transformation}

\newcommand{\alf}{\alpha}       \newcommand{\del}{\delta}   \newcommand{\Del}{\Delta}
\newcommand{\veps}{\varepsilon} \newcommand{\vphi}{\varphi} \newcommand{\gam}{\gamma}
\newcommand{\Gam}{\Gamma}          \newcommand{\lam}{\lambda}
\newcommand{\Lam}{\Lambda}          \newcommand{\Om}{\Omega}
\newcommand{\sig}{\sigma}          \newcommand{\tha}{\theta}
\newcommand{\vth}{\vartheta}    
     \newcommand{\Ups}{\Upsilon}

\newcommand{\mcA}{{\mathcal A}}   
   \newcommand{\mcH}{{\mathcal H}}

\DeclareMathOperator{\id}{{\mathsf{id}}}        
\DeclareMathOperator{\ed}{\lvert\cdot\rvert}    

\providecommand{\abs}[1]{\lvert#1\rvert}        

\DeclareMathOperator{\md}{\mathsf{mod}}

\DeclareMathOperator{\diam}{\mathsf{diam}}
\DeclareMathOperator{\dist}{\mathsf{dist}}

\newcommand{\B}{\mathsf{B}}     
\newcommand{\D}{\mathsf{D}}     
\newcommand{\A}{\mathsf{A}}     

\DeclareMathOperator{\Arg}{Arg}
\DeclareMathOperator{\Log}{Log}

\newcommand{\mathfont}{\mathsf} 
\newcommand{\mfC}{{\mathfont C}}      
\newcommand{\mfD}{{\mathfont D}}      
\newcommand{\mfN}{{\mathfont N}}      
\newcommand{\mfR}{{\mathfont R}}      
\newcommand{\mfS}{{\mathfont S}}      
\newcommand{\mfZ}{{\mathfont Z}}      

\newcommand{\cmfD}{\bar{\mathfont D}}      
\newcommand{\RS}{\hat{\mfC}}     
\newcommand{\Rh}{\hat{\mfR}}     
\newcommand{\bd}{\partial}      
\newcommand{\bA}{{\partial A}}    
\newcommand{\bD}{{\partial D}}    

\newcommand{\cA}{{\bar A}}      
	\newcommand{\bOm}{{\partial\Omega}} 

\newtheorem{Thm}{Theorem}
\renewcommand{\theThm}{\Alph{Thm}}      
\newtheorem{Cor}[Thm]{Corollary}        
\theoremstyle{remark}
\theoremstyle{plain}
\newtheorem*{thm*}{Theorem}         
\newtheorem*{lma*}{Lemma}           
\newtheorem*{cor*}{Corollary}
\newtheorem*{conj*}{Conjecture}
\newtheorem*{prop*}{Proposition}

\theoremstyle{remark}
\newtheorem*{claim*}{Claim}
\newtheorem*{xx*}{Example}
\newtheorem*{xxs*}{Examples}
\newtheorem*{fact*}{Fact}
\newtheorem*{qstn*}{Question}
\newtheorem*{rmk*}{Remark}
\newtheorem*{rmks*}{Remarks}

\swapnumbers                
\theoremstyle{plain}
\newtheorem{thm}[equation]{Theorem}
\newtheorem{lma}[equation]{Lemma}
\newtheorem{cor}[equation]{Corollary}
\newtheorem{prop}[equation]{Proposition}

\theoremstyle{remark}

\newtheorem{xx}[equation]{Example}

\newtheorem{fact}[equation]{Fact}
\newtheorem{facts}[equation]{Facts}
\newtheorem{qstn}[equation]{Question}

\newtheorem{rmks}[equation]{Remarks}

  {\par\smallskip\noindent{\em #1}}{\par\smallskip}

\newenvironment{noname}[1]
  {\par\smallskip\noindent%
   \leftskip=\nnlen\rightskip=\nnlen\addtolength{\leftmargini}{\nnlen}%
   \em #1}%
  {\par\smallskip\addtolength{\leftmargini}{-\nnlen}}
\newlength{\nnlen}

\newenvironment{pf}[1]
  {\par\smallskip\noindent\refstepcounter{equation}\theequation.{ \em #1.}}%
  {\qed\smallskip}
\newenvironment{pf*}[1]{\subsubsection*{#1}}{\qed\smallskip}	

\newcounter{aenumctr} 
  {\begin{list}%
    {\rm(\alph{aenumctr})}
    {\usecounter{aenumctr}}
    \setlength{\rightmargin}{\leftmargin}}
  {\end{list}} 

\renewcommand{\mathfont}{\mathbb}

\newcommand{\cra}{\curvearrowright}      

\newcommand{\lsig}{l_\sig}               
\newcommand{\ells}{\ell_\sig}            
\newcommand{\ksig}{k_\sig}
\newcommand{\kchi}{k_\chi}

\newcommand{\Ah}{\hat{A}}

\newcommand{\bOmh}{\hat{\bd}\Om}

\newcommand{\subann}{\mathrel{\ooalign{$\subset$\cr\hidewidth\raisebox{0.22ex}{\hspace{1mm}$\scriptstyle\rm a$}\hidewidth}}}

\newcommand{\csubann}{\mathrel{\ooalign{$\subset$\cr\hidewidth\raisebox{0.22ex}{\hspace{1mm}$\scriptstyle\rm c$}\hidewidth}}}

\newcommand{\Sone}{\mathsf{S}^1}    

\newcommand{\BL}{bi-Lipschitz}

\newcommand{\bp}{\mathsf{bp}}
\newcommand{\BP}{\mathsf{BP}}
\newcommand{\BPt}{\textsf{BP}}

\newcommand{\bdi}{\bd_{\rm in}}     
\newcommand{\bdo}{\bd_{\rm out}}    

\newcommand{\Ain}{A_{\rm in}}       
\newcommand{\Aout}{A_{\rm out}}     
\newcommand{\Piin}{\Pi_{\rm in}}    
\newcommand{\Piout}{\Pi_{\rm out}}  

\newcommand{\Coo}{{\mfC}_{01}}		
\newcommand{\Cab}{{\mfC}_{ab}}		
\newcommand{\Cstar}{{\mfC}_\star}		
\newcommand{\Dstar}{{\mfD}_\star}		
\newcommand{\Delstar}{{\Del}^{\!\star}}	

\newcommand{\kk}{\mathsf{k}}        
\newcommand{\cc}{\mathsf{c}}        

\begin{document} 
\title[Hyperbolic versus QuasiHyperbolic Distance]{Hyperbolic Distance versus QuasiHyperbolic Distance\\ in Plane Domains}
\date{\today}


\author{David A. Herron}
\address{Department of Mathematical Sciences, University of Cincinnati, OH 45221-0025, USA}
\email{David.Herron@UC.edu}

\author{Jeff Lindquist}
\address{Department of Mathematical Sciences, University of Cincinnati, OH 45221-0025, USA}
\email{jlindquistmath@gmail.com}




\keywords{hyperbolic metric, quasihyperbolic metric, quasisymmetry, quasiisometry}
\subjclass[2010]{Primary: 30F45, 30L99; Secondary: 51F99, 30C62}

\begin{abstract}
We examine Euclidean plane domains with their hyperbolic or quasihyperbolic distance.  We prove that the associated metric spaces are quasisymmetrically equivalent \ifff they are \BL\ equivalent.  On the other hand, for Gromov hyperbolic domains, the two corresponding Gromov boundaries are always quasisymmetrically equivalent.  Surprisingly, for any finitely connected hyperbolic domain, these two metric spaces are always quasiisometrically equivalent.  We construct an example where the spaces are not quasiisometrically equivalent.
\end{abstract}

\dedicatory{Dedicated to David Minda, for decades of interesting discussions.}

\newcommand{\PV}{{{\small \em  preliminary version---please do not circulate}}}
\maketitle
%

\section{Introduction}  \label{S:Intro} 
Throughout this section $\Om$ denotes a hyperbolic plane domain: $\Om\subset\mfC$ is open and connected and $\mfC\sm\Om$ contains at least two points.  Each such $\Om$ carries a unique maximal constant curvature -1 conformal metric $\lam\,ds=\lam_\Om\,ds$ usually referred to as the \emph{Poincar\'e hyperbolic metric} on $\Om$.  The length distance $h=h_\Om$ induced by $\lam\,ds$ is called \emph{hyperbolic distance} in $\Om$.  There is also a \emph{quasihyperbolic metric} $\del^{-1}ds=\del_\Om^{-1}ds$ on $\Om$, whose length distance $k=k_\Om$ is called \emph{quasihyperbolic distance} in $\Om$; here $\del(z)=\del_\Om(z):=\dist(z,\bOm)$ is the Euclidean distance from $z$ to the boundary of $\Om$.  See \rf{s:cfml metrics} for more details.

This works continues that begun in \cite{HB-geodesics}, \cite{DAH-k-univ-covers}, \cite{DAH-hlexj} where we elucidate the geometric similarities and metric differences between the metric spaces $(\Om,h)$ and $(\Om,k)$.  Our first result, \rf{TT:QSequiv} below, characterizes when the metric spaces $(\Om,h)$ and $(\Om,k)$ are \qsc ally equivalent.  Then \rf{TT:Gromov} reveals that, when these spaces are Gromov hyperbolic, their Gromov boundaries are always \qsc ally equivalent.

To set the stage, we begin with some preliminary observations.
A straightforward, albeit non-trivial, argument reveals that the metric spaces $(\Om,h)$ and $(\Om,k)$ are isometric \ifff $\Om$ is an open half-plane and the isometry is the restriction of a M\"obius transformation.  Furthermore, these metric spaces are bi-Lipschitz equivalent \ifff the identity map is bi-Lipschitz; see \rf{ss:BP}.

It is well-known that the identity map $(\Om,k)\xra{\id}(\Om,h)$ enjoys the following properties:
\begin{itemize}
  \item  The map $\id$ is a 2-Lipschitz $1$-quasiconformal \homeo.
  \item  For any simply connected $\Om$, $\id$ is 2-bi-Lipschitz.\footnote{That $\id^{-1}$ is 2-Lipschitz is a consequence of Koebe's One Quarter Theorem.}
  \item  In general, $\id$ is bi-Lipschitz \ifff $\RS\sm\Om$ is uniformly perfect.\footnote{This is quantitative: the bi-Lipschitz and uniformly perfectness constants depend only on each other.}
\end{itemize}
The last item above is due to Beardon and Pommerenke; see \cite{BP-beta} and \rf{ss:BP}.

For general hyperbolic plane domains $\Om$, there is no simple metric control on $\id^{-1}$.  For example, given \emph{any} sequences $(h_n)_1^\infty$ and $(k_n)_1^\infty$ of positive numbers with say $1\ge h_n\to0$ and $2\le k_n\to\infty$, there are sequences $(a_n)_1^\infty, (b_n)_1^\infty$ of points in the punctured unit disk $\mfD_\star:=\mfD\sm\{0\}$ with hyperbolic and quasihyperbolic distances $h_\star(a_n,b_n)=h_n$ and $k_\star(a_n,b_n)=k_n$.  See \cite[Ex.~2.7]{HB-geodesics}

The following striking rigidity theorem contains our first main result.  (See \rf{ss:MPG} for mapping definitions.)  This says that  the metric spaces $(\Om,k)$ and $(\Om,h)$ are either ``quite similar'' (i.e., \BL\ equivalent) or ``quite different'' (i.e., not \qsc ally equivalent).

\begin{Thm}  \label{TT:QSequiv}
For any hyperbolic plane domain $\Om$, \tfaqe.
\begin{enumerate}[\rm(\theThm.1), wide, labelwidth=!]
  \item  The metric spaces $(\Om,k)$ and $(\Om,h)$ are \qsc ally equivalent.
  \item  The metric spaces $(\Om,k)$ and $(\Om,h)$ are \BL\ equivalent.
  \item  The identity map $(\Om,k)\to(\Om,h)$ is \BL.
  \item  $\RS\sm\Om$  is uniformly perfect.
\end{enumerate}
\end{Thm}                      
\noindent
Again, Beardon and Pommerenke \cite{BP-beta} established the equivalence of (\ref{TT:QSequiv}.3) and (\ref{TT:QSequiv}.4).

Recently, the first author and Buckley \cite[Theorem~B]{HB-geodesics} demonstrated that $(\Om,k)$ and $(\Om,h)$ are simultaneously Gromov hyperbolic or not, and we call $\Om$ Gromov hyperbolic in the former case.  Our next result stands in stark contrast to \rf{TT:QSequiv}.  Even when $(\Om,k)$ and $(\Om,h)$ are not \qsc ally equivalent, the large scale geometry is the same in both spaces---at least when they are Gromov hyperbolic.  This further enhances \cite[Theorem~A]{HB-geodesics} where we proved that these metric spaces have the same quasi-geodesic curves.

\begin{Thm}  \label{TT:Gromov} 
For any Gromov hyperbolic plane domain $\Om$, the canonical conformal gauges on the Gromov boundaries $\bd_G(\Om,k)$ and $\bd_G(\Om,h)$ are naturally \qsc ally equivalent. \flag{quantitative too, right? at least sorta}
\end{Thm}                      

When $(\Om,h)$ and $(\Om,k)$ are \BL\ equivalent, or just quasiisometrically\footnote{%
Our quasiisometries are sometimes called rough \BL\ maps.}
equivalent, the above Gromov boundary equivalency is given via a power \qsy.  Whereas our proof of \rf{TT:QSequiv} is surprisingly simple, the proof of \rf{TT:Gromov} employs significant machinery as explained in the first paragraph of \rf{s:pfThmGromov}.

Again, when $\RS\sm\Om$ is uniformly perfect, $(\Om,k)$ and $(\Om,h)$ are \BL\ equivalent.  A natural conjecture is that this uniform perfectness might be a neccessary condition for $(\Om,k)$ and $(\Om,h)$ to be quasiisometrically equivalent.  However, our next result (which, initially, was unexpected to these authors) reveals that this is not the case.

\begin{Thm}  \label{TT:isolated} 
For any finitely connected hyperbolic plane domain $\Om$, the metric spaces $(\Om,k)$ and $(\Om,h)$ are quasiisometrically equivalent.
\end{Thm}                        

The above is just an easy to state special case of our more general \rf{T:UP+S} which provides a large class of plane domains whose hyperbolizations and quasihyperbolizations are \qic ally equivalent.  This raises the natural question of whether the conclusion of \rf{TT:isolated} could be true in general, and we answer this below.  Note that a quasiisometry can have an arbitrarily large additive rough constant and this obstacle must be overcome.

\begin{Thm}  \label{TT:example} 
There is a uniform (hence Gromov) hyperbolic plane domain $\Om$ with the property that any \qsc\ equivalence between $\bd_G(\Om,k)$ and $\bd_G(\Om,h)$, e.g., that given by \rf{TT:Gromov}, is \emph{not} via a \emph{power} \qsy.  In particular,  $(\Om,k)$ and $(\Om,h)$ are \emph{not} quasiisometrically equivalent.
\end{Thm}                       

\medskip

\rf{S:Prelims} contains the usual definitions and terminology; especially, see \rf{ss:quasihyp metric} and \rf{ss:hyp metric} for details about the hyperbolic and quasihyperbolic metrics.  We prove \rfs{TT:QSequiv}, \ref{TT:Gromov}, \ref{TT:isolated}, \ref{TT:example} in \rfs{s:pfThmQSequiv}, \ref{s:pfThmGromov}, \ref{s:pfThmIsolated}, \ref{s:example} respectively. 
%
\section{Preliminaries}  \label{S:Prelims} 
%
We work in the Euclidean plane, and on the Riemann sphere, which we identify, respectively, with the complex number field $\mfC$ and its one-point extension $\RS:=\mfC\cup\{\infty\}$.  Everywhere $\Om$ is a domain (i.e., an open connected set) and $\bOm$ and $\hat\bOm$ denote the boundary of $\Om$ \wrt the plane and sphere (respectively)   Always, $\Om$ is a \emph{hyperbolic domain}, i.e., $\RS\sm\Om$ contains at least three points.

We write $\Ah$ and $\hat{\bd}A$ for the closure and boundary of a set $A$ in $\RS$ whereas $\cA$ and $\bA$ are the Euclidean closure and boundry of $A$.

We write $C=C(D,\dots)$ to indicate a constant $C$ that depends only on the data $D,\dots$.  In some cases we write $K_1\lex K_2$ to indicate that $K_1\le C\,K_2$ for some computable constant $C$ that depends only on the relevant data, and $K_1\eqx K_2$ means $K_1\lex K_2\lex K_1$.

The Euclidean line segment joining two points $a,b$ is $[a,b]$, and $(a,b)=[a,b]\sm\{a,b\}$.  The open and closed Euclidean disks, and the circle, centered at the point $a\in\mfC$ and of radius $r>0$, are denoted by $\D(a;r)$ and $\D[a;r]$ and $\mathsf{S}^1(a;r)$ respectively, and $\mfD:=\D(0;1)$ is the unit disk.  We also define
\[
  \Cstar:=\mfC\sm\{0\} \,,\quad \Cab:=\mfC\sm\{a,b\} \,,\quad \Dstar:=\mfD\sm\{0\} \,, \quad \mfD^\star:=\mfC\sm\cmfD\,;
\]
the definition of $\Cab$ is for distinct points $a$, $b$ in $\mfC$.

The quantity $\del(z)=\del_\Om(z):=\dist(z,\bOm)$ is the Euclidean distance from $z\in\mfC$ to the boundary of $\Om$, and $1/\del$ is the scaling factor (aka, metric-density) for the so-called \emph{quasihyperbolic} metric $\del^{-1}ds$ on $\Om\subset\mfC$; see \rf{ss:quasihyp metric}.  We use the notation
\begin{gather*}
  \D(z)=\D_{\Om}(z):=\D\bigl(z;\del(z)\bigr)=\D\bigl(z;\del_\Om(z)\bigr)
  \intertext{for the maximal Euclidean disk in $\Om$ centered at a point $z\in\Om$, and then}
  \B(z)=\B_\Om(z):=\bd\D(z)\cap\bOm = \Sone\bigl(z;\del(z)\bigr)\cap\bOm
\end{gather*}
is the set of all nearest boundary points for $z$.

The \emph{chordal} and \emph{spherical} distances on $\RS$ are $\chi$ and $\sig$, respectively.  Thus
\[
  \chi(z,w):=\begin{cases}
    \ds \frac{2|z-w|}{\sqrt{1+|z|^2}\sqrt{1+|w|^2}} \quad&\text{if $z,w\in\mfC$} \\
    \ds \frac{2}{\sqrt{1+|z|^2}} &\text{if $z\in\mfC,w=\infty$}
  \end{cases}
\]
and $\sig$ is the length distance associated with $\chi$, $\chi=2\sin(\sig/2)$, and $\chi\le\sig\le(\pi/2)\chi$.\footnote{This follows at once when we identify $\RS$ with the unit sphere in $\mfR^3$ where we find that $\sig(u,v)$ is just the angle between $u,v$.}
Calculations are easier with $\chi$, but  $\sig$ is a geodesic distance whereas $\chi$ is not geodesic.

Each of the metric spaces $(\Om,\ed),(\Om,\chi),(\Om,\sig)$ has an associated length distance, although the latter two are equal.  We write $l=l_\Om$ for the intrinsic (aka, inner) Euclidean length distance and $\lsig=l_{\Om_\sig}$ for the intrinsic chordal length distance (which equals the intrinsic spherical length distance:-), and then $(\Om,l)$ and $(\Om,\lsig)$ are the corresponding length spaces.  See for example \cite{DAH-diam-dist}.

It is convenient to let $\chi(z)$ and $\sig(z)$ denote the chordal and spherical distances from $z$ to $\hat{\bd}\Om$.  Again, $\chi(z)\le\sig(z)\le(\pi/2)\chi(z)$, and we note that
\[
  \sig(z):=\dist_\sig(z,\bOmh)=\dist_{\lsig}\bigl(z,\bd(\Om,\lsig)\bigr)
\]
where $\bd(\Om,\lsig)$ is the metric boundary of $(\Om,\lsig)$.

\subsection{Maps, Paths, and Geodesics}     \label{ss:MPG} %
An embedding $X\xra{f}Y$ between two metric spaces is a \emph{\qsy} if there is a \homeo\ $\eta:[0,\infty)\to[0,\infty)$ (called a \emph{distortion function}) such that for all triples $x,y,z\in X$,
\[
  |x-y|\le t |x-z| \implies |fx-fy| \le \eta(t) |fx-fz|\,;
\]
when this holds, we say that $f$ is $\eta$-QS.  These mappings were studied by Tukia and V\"ais\"al\"a in \cite{TV-qs}; see also \cite{Juha-analysis}.

The \emph{\BL} maps form an important subclass of the \qsc\ maps; $X\xra{f}Y$ is \BL\ \ifff there is a constant $L$ such that for all $x,y\in X$,
\[
  L^{-1} |x-y|\le |fx-fy|\le L|x-y|,
\]
and when this holds we say that $f$ is $L$-\BL; such an $f$ is $\eta$-QS with $\eta(t):=Lt$.


More generally, a map $X\xra{f}Y$ is an \emph{$(L,C)$-\qiy} if $L\ge1$, $C\ge0$ and for all $x,y\in X$,
\[
  L^{-1} |x-y| -C \le |fx-fy| \le L|x-y| + C.
\]
These are often called \emph{rough \BL} maps, and there seems to be no universal agreement regarding this terminology; some authors use the adjective \qiy\ to mean what we have called \BL, and then a rough \qiy\ satisfies our definition of \qiy.  A $(1,0)$-\qiy\ is simply an \emph{isometry} (onto its range), and a $(1,C)$-\qiy\ is called a \emph{$C$-rough isometry}.

Two metric spaces $X,Y$ are \emph{isometrically equivalent} (or \emph{BL, QS, QC equivalent}, respectively) \ifff there is a bijection $X\to Y$ that is an isometry (or BL, QS, or QC).

Also, $X,Y$ are \emph{quasiisometrically equivalent} \ifff there is a \qiy\ $X\xra{f}Y$ with the property that $f(X)$ is \emph{cobounded} in $Y$ (i.e., the Hausdorff distance between $f(X)$ and $Y$ is finite).  More precisely: $X,Y$ are \emph{$(L,C)$-QI equivalent} if there is an $(L,C)$-\qiy\ $f:X\to Y$ and for each $y\in Y$ there is an $x\in X$ with $|y-f(x)|\le C$.  An alternative way to describe this is to say that there are quasiisometries in both directions that are rough inverses of each other. 

\begin{xx}  \label{X:QIEqX} %
Suppose $X\subset Y$ and for each $y\in Y$ there is an $x_y\in X$ with $|x_y-y|\le C$.  The maps
    \[
      X\overset{\id}{\hookrightarrow}Y\xra{f}X\,, \quad\text{where}\quad f(y):=
        \begin{cases}
          y \quad&\text{if $y\in X$}\\
          x_y &\text{if $y\notin X$}
        \end{cases}\,,
    \]
    are both rough isometric equivalences: the ``identity'' inclusion is a $(1,C)$-QI equivalence and $f$ is a $(1,2C)$-QI equivalence.
\end{xx} %

Our metric spaces will always be the domain $\Om$, either in $\mfC$ or in $\RS$, with either Euclidean distance, chordal distance, spherical distance, an associated length distance, an associated quasihyperbolic distance, or an associated hyperbolic distance.

\medskip

A \emph{path in $X$} is a continuous map $\mfR\supset I\xra{\gam}X$ where $I=I_\gam$ is an interval (called the \emph{parameter interval for $\gam$}) that may be closed or open or neither and finite or infinite.  The \emph{trajectory} of such a path $\gam$ is $|\gam|:=\gam(I)$ which we call a \emph{curve}.  When $I$ is closed and $I\ne\mfR$, $\bd\gam:=\gam(\bd I)$ denotes the set of \emph{endpoints of $\gam$} which consists of one or two points depending on whether or not $I$ is compact.  For example, if $I_\gam=[0,1]\subset\mfR$, then $\bd\gam=\{\gam(0),\gam(1)\}$.

We call $\gam$ a \emph{compact path} if its parameter interval $I$ is compact (which we often assume to be $[0,1]$).  We call $\gam$ a \emph{rectifiable path} if its length $\ell(\gam)$ is finite, and then we may assume that $\gam$ is parameterized \wrt arclength in which case the parameter interval for $\gam$ is $[0,\ell(\gam)]$.  We note that arclength \param s are \emph{a priori} $1$-Lipschitz continuous.

When $\bd\gam=\{a,b\}$, we write $\gam:a\curvearrowright b$ (\emph{in $\Om$}) to indicate that $\gam$ is a path (in $\Om$) with \emph{initial point} $a$ and \emph{terminal point} $b$; this notation implies an orientation: $a$ precedes $b$ on $\gam$.

When $\alf:a\cra b$ and $\beta:b\cra c$ are paths that join $a$ to $b$ and $b$ to $c$ respectively, $\alf\star\beta$ denotes  the concatenation of $\alf$ and $\beta$; so $\alf\star\beta:a \cra c$.  Of course, $|\alf\star\beta|=|\alf|\cup|\beta|$.
Also, the \emph{reverse of $\gam$} is the path $\gam^{-1}$ defined by $\gam^{-1}(t):=\gam(1-t)$ (when $I_\gam=[0,1]$) and going from $\gam(1)$ to $\gam(0)$.  Of course, $|\gam^{-1}|=|\gam|$.

An \emph{arc} $\alf$ is an injective compact path.  Every arc is taken to be ordered from its initial point to its terminal point.  Given points $a,b\in|\alf|$, there are unique $u,v\in I$ with $\alf(u)=a$, $\alf(v)=b$ and we write $\alf[a,b]:=\alf\vert_{[u,v]}$. 
Every compact path contains an arc with the same endpoints; see \cite{V-exhsJohn}.

\medskip

A path $I\xra{\gam}X$ into a metric space $X$ is a \emph{geodesic} if $\gam$ is an isometry (for all $s,t\in I$, $|\gam(s)-\gam(t)|=|s-t|$) and a \emph{$K$-quasi-geodesic} if $\gam$ is $K$-\BL,\footnote{%
One can also consider rough-quasi-geodesics where $\gam$ is a quasiisometry; we do not use these here.}
\[
  \text{for all $s,t\in I$}\,, \quad K^{-1}|s-t| \le |\gam(s)-\gam(t)| \le K|s-t|.
\]
A characteristic property of geodesics is that the length of each subpath equals the distance between its endpoints.  There is a corresponding description for quasi-geodesics: $I\xra{\gam}X$ is an \emph{$L$-chordarc path} if it is rectifiable and
\[
  \text{for all $s,t\in I$}\,, \quad \ell(\gam\vert_{[s,t]}) \le L\,|\gam(s)-\gam(t)|.
\]
If we ignore parameterizations, then the class of all quasi-geodesics (in some metric space) is exactly the same as the class of all chordarc paths.  More precisely, a $K$-quasi-geodesic is a $K^2$-chordarc path, and if we parameterize an $L$-chordarc path \wrt arclength, then we get an $L$-quasi-geodesic.

\smallskip

In this paper we study the metric spaces $(\Om,h)$ or $(\Om,k)$ where $\Om$ is a hyperbolic plane domain and $h$ and $k$ are the hyperbolic and quasihyperbolic distances in $\Om$.  The geodesics and quasi-geodesics in $(\Om,h)$ are called \emph{hyperbolic} geodesics and \emph{hyperbolic} quasi-geodesics, and similarly in $(\Om,k)$ we attach the adjective \emph{quasihyperbolic}.

\subsection{Annuli and Uniformly Perfect Sets}      \label{s:A&UPS} %
Given a point $c\in\mfC$ and $0<r<R<+\infty$, 
$A:=\Set{z\in\mfC | r<\abs{z-c}<R}$ is an \emph{Euclidean annulus} with \emph{center} $c(A):=c$ and \emph{conformal modulus} $\md(A):=\log(R/r)$; if $r=0$ or $R=\infty$, $A$ is a \emph{degenerate annulus} and  $\md(A):=+\infty$.
We call $\mathsf{S}^1(A):=\mathsf{S}^1(c;\sqrt{rR})$ the \emph{conformal center circle} of $A$; $A$ is symmetric about this circle.  The \emph{inner} and \emph{outer boundary circles} of $A$ are, respectively,
\[
  \bdi A:=\mathsf{S}^1(c;r) \quad\text{and}\quad  \bdo A:=\mathsf{S}^1(c;R)\,.
\]
A point $z$ is \emph{inside} (\emph{outside}) $A$ \ifff $z\in\Ain:=\B[c;r]$ ($z\in\Aout:=\mfC\sm\B(c;R)$); that is, $z$ is inside (or outside) $A$ \ifff $z$ is inside $\bdi A$ (or outside $\bdo A$).

It is convenient to introduce the notation
\[
  A=\A(c;d,m):=\Set{z\in\mfC : d\,e^{-m}<|z-c|<d\,e^m}
  \quad\text{and}\quad
  \A[c;d,m]:=\overline{\A(c;d,m)}\,.
\]
Then $\mathsf{S}^1(A)=\mathsf{S}^1(c;d), \bdi A=\mathsf{S}^1(c;d\,e^{-m}), \bdo A=\mathsf{S}^1(c;d\,e^m)$; here $c=c(A), d>0, m>0$.

An annulus $A'$ is a \emph{subannulus of $A$}, denoted by $A'\subann A$, provided
\[
  A'\subset A \quad\text{and}\quad  \Ain\subset\Ain' \quad\text{and}\quad  \Aout\subset\Aout'\,.
\]
Two annuli are \emph{concentric} if they have a common center, and $A'$ is a \emph{concentric subannulus of $A$}, denoted by $A'\csubann A$, provided $c(A')=c(A)$ and $A'$ is a subannulus of $A$.

An annulus $A$ \emph{separates} $E$ if $A\subset\RS\sm E$ and both components of $\RS\sm A$ contains points of $E$; thus when $A$ separates $\Set{a,b}$, one of $a$ or $b$ lies inside $A$ and the other lies outside $A$, and if $A$ does not meet nor separate $\Set{a,b}$, then $a$ and $b$ are \emph{on the same side} of $A$.  Evidently, if $A'$ is a subannulus of $A$, then $A'$ separates the boundary circles of $A$.

We define
\begin{align*}
  \mcA(m)      &:=\Set{A | \text{$A$ is an Euclidean annulus with $\md(A)>m$}}\,, \\
  \mcA_\Om     &:=\Set{A | \text{$A$ is an Euclidean annulus in $\Om$ with $c(A)\in\mfC\sm\Om$}}\,, \\
  \mcA^1_\Om   &:=\Set{A\in\mcA_\Om | \bA\cap\bOm\ne\emptyset}\,,  \\
  \mcA^2_\Om   &:=\Set{A\in\mcA_\Om | \bdi A\cap\bOm\ne\emptyset\ne\bdo A\cap\bOm}\,,  \\
  \mcA_\Om(m)     &:=\mcA_\Om\cap\mcA(m)\,, \;\text{and}\; \mcA_\Om^{s}(m):=\mcA^s_\Om\cap\mcA(m) \quad\text{for $s\in\Set{1,2}$}\,.
\end{align*}
The requirement $c(A)\in\mfC\sm\Om$ means that $A$ separates $(\mfC\sm\Om)\cup\Set{\infty}$.

Following Pommerenke \cite{Pomm-unifly-perfect1}, we say that $E\subset\RS$ is \emph{$M$-uniformly perfect} \ifff $E$ is closed, $E$ contains the point at infinity, and
\begin{equation}\label{E:bounded mod}
  \sup_{A\in\mcA_{\mfC\sm E}} \md(A) \le M \,.
\end{equation}
Pommerenke \cite{Pomm-unifly-perfect2} established a number of equivalent conditions; see also \cite{Sugawa-unif-perfect-geom}, \cite{HLM-sep-rings}.

Heinonen \cite{Juha-analysis} has a general metric space definition for uniform perfectness (which appears as item (\ref{P:UP equivs}.e) below) which is similar to, but different from, Pommerenke's definition.  Nonetheless, the following result shows that the Pommerenke and Heinonen definitions are equivalent; this must be folklore, but we do not know a reference.  Recall that a \emph{ring domain} is a topological annulus, and a ring domain $D\subset\RS$ \emph{separates} a set $E\subset\RS$ \ifff $D\subset\RS\sm E$ and both components of $\RS\sm D$ contain points of $E$.

\newcommand{\Rk}{\hat{R}}  
\newcommand{\rh}{\hat{r}}
\newcommand{\Mh}{\hat{M}}
\begin{prop} \label{P:UP equivs} %
For any closed set $E\subset\RS$ with $\infty\in E$, \tfaqe:
\begin{enumerate}[\rm(a), wide, labelwidth=!, labelindent=0pt]
  \item  $E$ is $M$-uniformly perfect,
  \item  $\bp\le M$ in $\Om$ for each component $\Om$ of $\RS\sm E$,
  \item  $\md(A)\le M$ for any annulus $A\subset\RS\sm E$ that separates $E$,
  \item  $\md(D)\le M$ for any ring domain $D\subset\RS\sm E$ that separates $E$,
  \item  $\Rk/\rh\le M$ for any $o\in E, \Rk>\rh>0$ with $\D_\chi(o;\Rk)\sm\bar\D_\chi[o;\rh]\subset\RS\sm E$ separating $E$.
\end{enumerate}
The constant $M$ above will vary, but in each case it depends only on the other constants.
\end{prop} 
\begin{proof}%
It is easy to check that (a) and (c) are equivalent.
Beardon and Pommerenke \cite{BP-beta} established the equivalence of (c) and (b), and Pommerenke \cite{Pomm-unifly-perfect1},\cite{Pomm-unifly-perfect2}  demonstrated that these are equivalent to (d) (along with several other variations).

That (e) is equivalent to the other conditions is surely well-known, yet does not seem to be explicitly mentioned in the literature, so we outline an explanation for this.

Suppose $o\in E$, $R>r>0$, and $D:=\D_\chi(o;R)\sm\bar\D_\chi[o;r]\subset\RS\sm E$ separates $E$.  Then $D$ is a ring domain, so if (d) holds, then
\[
  \log\frac{R}{r}\le\log\frac{R\sqrt{4-r^2}}{r\sqrt{4-R^2}}=\md(D)\le
 M\,.
\]

Finally, suppose (e) holds with a constant $ M\ge2$.  To establish (a), we verify that \eqref{E:bounded mod} holds with the constant $64M^4$.  Let $A=\Set{r<\abs{z-c}<R}\subset\mfC\sm E$ be an annulus with $c\in E\cap\mfC$; thus $A$ separates $E$.  Assume $R\ge2r$.  Roughly, we exhibit a chordal subannulus $\Ah\subann A$ and use $\Rk/\rh\le M$ to bound $R/r$.


There are three cases depending on whether $|c|\ge R$, $|c|\le r$, or $r<|c|<R$.  Assume $|c|\ge R$.  Here we verify that $R/r\le4M$.  Let
\begin{gather*}
  a:=\bigl(|c|-r\bigr)\frac{c}{|c|}\,,\; b:=\bigl(|c|+R\bigr)\frac{c}{|c|} \,,\; \rh:=\chi(a,c)\,,\; \Rk:=\chi(b,c)\,.
  \intertext{Since $|c|+R\le2|c|$ and $|c|-r\ge\half|c|$,}
  \frac{1+|b|^2}{1+|a|^2} = \frac{1+\bigl(|c|+R\bigr)^2}{1+\bigl(|c|-r\bigr)^2} \le  \frac{1+\bigl(2|c|\bigr)^2}{1+\bigl(|c|/2\bigr)^2} \le  16\,.
  \intertext{Thus if $\rh\ge\Rk$, then $R/r\le4$; so we assume $\rh<\Rk$.  Then}
  \Ah:=\Set{\hat{r}<\chi(z,c)<\hat{R}} \subann A
  \intertext{whence}
   M\ge\frac{\Rk}{\rh}=\frac{R}{r}\biggl(\frac{1+|a|^2}{1+|b|^2}\biggr)^\half\,, \quad\text{so}\quad \frac{R}{r}\le4 M\,.
\end{gather*}

Assume $|c|\le r$.  Here, eventually, we deduce that $R/r\le8M^2$.  Let
\begin{gather*}
  a:=\bigl(|c|+r\bigr)\frac{c}{|c|}\,,\; b:=\bigl(|c|-R\bigr)\frac{c}{|c|} \,,\; \rh:=\chi(a,0)\,,\; \Rk:=\chi(b,0)\,.
  \intertext{If $\rh\ge\Rk$, then $\chi(b,\infty)\ge\chi(a,\infty)$, so again $R/r\le4$.  Suppose $\rh<\Rk$.  Then}
  \Ah:=\Set{\hat{r}<\chi(z,c)<\hat{R}} \subann A\,, \quad\text{whence}\quad \frac{\Rk}{\rh} \le  M\,.
\intertext{We claim that}
  \sqrt{1+|a|^2}\ge\frac{\sqrt2  M}{2 M-1} \quad\text{and}\quad \sqrt{1+|b|^2}\le\sqrt2  M\,;
  \intertext{therefore $\ds M\ge\frac{\Rk}{\rh}=\frac{|b|}{|a|}\biggl(\frac{1+|a|^2}{1+|b|^2}\biggr)^\half$, so}
  \frac{R/2}{2r} \le \frac{R-|c|}{r+|c|} = \frac{|b|}{|a|} \le \biggl(\frac{1+|b|^2}{1+|a|^2}\biggr)^\half M \le M(2M-1)
\end{gather*}
and hence $R/r\le8M^2$.

To corroborate the above claim, we utilize chordal subannuli that have $|z|=1$ as a boundary circle.  If $\sqrt2\le\rh=\chi(a,0)$, then $|a|\ge1$ so $\sqrt{1+|a|^2}\ge\sqrt2M/(2 M-1)$.  Suppose $\rh<\sqrt2$.  Then
\begin{gather*}
  \Set{\rh<\chi(z,0)<\sqrt2}\subann A
  \intertext{so}
  M\ge\frac{\sqrt2}{\rh}=\frac{\sqrt{1+|a|^2}}{\sqrt2|a|}\le\frac{1+|a|}{2|a|}
  \intertext{and therefore}
  \sqrt{1+|a|^2}\ge\frac{1+|a|}{\sqrt2}\ge\frac{\sqrt2M}{2M-1}\,.
\end{gather*}
If $\sqrt2\le\rh_\infty:=\chi(b,\infty)$, then $\sqrt{1+|b|^2}\le\sqrt2\le\sqrt2M$.  Suppose $\rh_\infty<\sqrt2$.  Then
\begin{gather*}
  \Set{\rh_\infty<\chi(z,0)<\sqrt2}\subann A
\end{gather*}
so $\ds  M\ge\frac{\sqrt2}{\rh}=\frac{\sqrt{1+|b|^2}}{\sqrt2}$.

\medskip

Thus when $|c|\le r$ or $|c|\ge R$, $R/r\le 8M^2$.  Suppose $r<|c|<R$, and $R\ge4r$.  Then either $|c|\ge\sqrt{rR}$ or $|c|\le\sqrt{rR}$.  In either case, we can apply the previous arguments to the appropriate subannulus (either
$\{r<|z-c|<\sqrt{Rr}\}$ or $\{\sqrt{Rr}<|z-c|<R\}$) to conclude that in all cases $R/r\le 64 M^4$.
\end{proof}%

\subsection{Conformal Metrics}      \label{s:cfml metrics} %
A continuous function $X\xra{\rho}(0,\infty)$ on a rectifiably connected metric space $X$ induces
a length distance $d_\rho$ on $X$ defined by
\[
  d_\rho(a,b):=\inf_{\gam:a\cra b} \ell_\rho(\gam) \quad\text{where}\quad \ell_\rho(\gam):= \int_\gam \rho\, ds
\]
and where the infimum is taken over all rectifiable paths $\gam:a\cra b$ in $X$.  We describe this by calling $\rho\,ds=\rho(x)|dx|$ a \emph{conformal metric} on $X$.  Below we consider the hyperbolic and quasihyperbolic metrics defined on plane domains.

We call $\gam$ a \emph{$\rho$-geodesic} if $d_\rho(a,b)=\ell_\rho(\gam)$; these need not be unique.  We often write $[a,b]_\rho$ to indicate a $\rho$-geodesic with endpoints $a,b$, but one must be careful with this notation since these geodesics need not be unique.  When $z$ is a fixed point on a given geodesic $[a,b]_\rho$, we write $[a,z]_\rho$ to mean the subarc of the given geodesic from $a$ to $z$.

We note that the ratio $\rho\,ds/\sig\,ds$ of two conformal metrics is a well-defined positive function.  We write $\rho\le C\,\sig$ to indicate that this metric ratio is bounded above by $C$.

Wen $\rho\,ds$ is a conformal metric on $\Om$, we let $\Om_\rho:=(\Om,d_\rho)$.  The following is surely folklore, but we briefly outline a proof which employs standard techniques.

\begin{lma} \label{L:cfml metrics} %
Let $\rho\,ds$ and $\tau\,ds$ be conformal metrics on some plane domain $\Om$.  Then the identity map
\[
  \Om_\rho\xra{\id}\Om_\tau \quad\text{is conformal, i.e., it is metrically $1$-QC}.
\]
Also, if $\Om_\rho\xra{f}\Om_\tau$ is $\eta$-QS, then $f$ induces a map $\Om\xra{f}\Om$ which is $\eta(1)$-QC.
\end{lma} 
\begin{proof}%
The metric spaces $(\Om,\ed),(\Om,l),(\Om,d_\rho),(\Om,d_\tau)$ are all \homic, and even locally \BL.  Indeed, given $a\in\Om$ and $0<\veps<\rho(a)/2$, there is an $r_0\in(0,\del(a))$ such that
\begin{equation}\label{E:dr locBL}
  z\in\D[a;r_0] \implies
  \bigl(\rho(a)-\veps\bigr)|z-a| \le d_\rho(z,a) \le \bigl(\rho(a)+\veps\bigr)|z-a|\,.
\end{equation}
Using this we deduce that for all small  $\veps>0$ and all $r\in(0,r_0)$ (where $r_0\in(0,\del(a))$ is chosen as above ``for both $d_\rho$ and $d_\tau$''),
\begin{gather*}
  L(r):=\sup_{d_\rho(z,a)\le r} d_\tau(z,a) \le \frac{\tau(a)+\veps}{\rho(a)-\veps} \cdot r
  \intertext{and}
  l(r):=\inf_{d_\rho(z,a)\ge r} d_\tau(z,a) \ge \frac{\tau(a)-\veps}{\rho(a)+\veps} \cdot r
  \intertext{whence}
  \limsup_{r\to0^+} \frac{L(r)}{l(r)} \le
  \frac{\tau(a)+\veps}{\tau(a)-\veps}\cdot\frac{\rho(a)+\veps}{\rho(a)-\veps}
  \to1\;\text{as $\veps\to0^+$}\,.
\end{gather*}

Suppose $\Om_\rho\xra{f}\Om_\tau$ is $\eta$-QS, $a\in\Om$, and $\veps>0$.  Since $f$ is a \homeo, for all $r\in(0,\del(a))$
\begin{gather*}
  L_f(r):=\sup_{|z-a|\le r} |f(z)-f(a)|=\max_{|z-a|=r} |f(z)-f(a)|
  \intertext{and}
  l_f(r):=\inf_{|z-a|\ge r} |f(z)-f(a)|=\min_{|z-a|=r} |f(z)-f(a)|\,.
\end{gather*}
Below we write $z':=f(z),a':=f(a)$, etc.

As above, pick $r_0\in(0,\del(a))$ and $s_0\in(0,\del(a'))$ so that \eqref{E:dr locBL} holds for $d_\rho$ and its analog holds for $d_\tau$ and $w\in\D[a';s_0]$.  Then take $r_1\in(0,r_0]$ so that $f\bigl(\D[a;r_1]\bigl)\subset\D(a';s_0)$.  Then for all $r\in(0,r_1)$ and all $z,w\in\mathsf{S}^1(a;r)$,
\begin{gather*}
  \frac{|z'-a'|}{|w'-a'|} \le \frac{\tau(a')+\veps}{\tau(a')-\veps} \cdot \frac{d_\tau(z',a')}{d_\tau(w',a')}
  \le \frac{\tau(a')+\veps}{\tau(a')-\veps} \cdot \eta\biggl(\frac{d_\tau(z',a')}{d_\tau(w',a')}\biggr) \le \frac{\tau(a')+\veps}{\tau(a')-\veps} \cdot \eta\biggl(\frac{\rho(a)+\veps}{\rho(a)-\veps}\biggr)\,;
  \intertext{here the inner inequality holds by \qsy\ and the two outer inequalities follow from repeated applications of \eqref{E:dr locBL}.  Selecting such $z,w$ that attain $L_f(r),l_f(r)$ respectively yields}
  \frac{L_f(r)}{l_f(r)} = \frac{|z'-a'|}{|w'-a'|} \le   \frac{\tau(a')+\veps}{\tau(a')-\veps} \cdot \eta\biggl(\frac{\rho(a)+\veps}{\rho(a)-\veps}\biggr)
\end{gather*}
and so letting $r\to0^+$, then $\veps\to0^+$, we deduce that, \wrt Euclidean distance, $f$ is indeed $\eta(1)$-QC.
\end{proof}%

The careful reader recognizes that in the above, we employ the metric (aka, linear) dilatation for \qc\ maps, not the geometric dilatation, whereas in \rf{F:GO-k-RQI} below $K$ is the geometric dilatation.

\subsubsection{The QuasiHyperbolic Metric}    \label{ss:quasihyp metric} %
The \emph{quasihyperbolic metric} $\del^{-1}ds$ is defined for any proper subdomain $\Om\subsetneq\mfC$; here $\del=\del_\Om$ is the Euclidean distance to the boundary of $\Om$.  This metric can be defined in very general metric spaces and has proven useful in many areas of geometric analysis.  See \cite{BHK-unif} and \cite{HRS-HdfCvgcQhypDist}.

For domains $\Om\subsetneq\RS$, we also consider the \emph{chordal quasihyperbolic metric} $\chi^{-1}d\hat{s}$ and the \emph{spherical quasihyperbolic metric} $\sig^{-1}d\hat{s}$ where $d\hat{s}$ denotes the chordal (or spherical) arclength ``differential''.  (Recall the paragraph immediately preceding \rf{ss:MPG}.)  The latter was employed in \cite[Chapter~7]{BHK-unif}.

The \emph{Euclidean, chordal, and spherical quasihyperbolic distances} $k=k_\Om, \kchi=k_{(\Om,\chi)}$ and $\ksig=k_{(\Om,\sig)}$ in $\Om$ are the length distances induced by the Euclidean quasihyperbolic, chordal quasihyperbolic, and spherical quasihyperbolic metrics $\del^{-1}\,ds, \chi^{-1}\,ds$ and $\sig^{-1}ds$ on $\Om$; here $\Om\subsetneq\mfC$ in the former setting whereas $\Om\subsetneq\RS$ in the latter two.  These are geodesic distances.

The Euclidean length space $(\Om,l)$ and spherical length space $(\Om,l_\sig)$ also carry quasihyperbolic metrics, but: quasihyperbolic distance in $(\Om,l)$ is Euclidean quasihyperbolic distance and quasihyperbolic distance in $(\Om,\lsig)$ is spherical quasihyperbolic distance.

We remind the reader of the following basic estimates for quasihyperbolic distance, first established by Gehring and Palka \cite[2.1]{GP-qch}:  For all $a,b\in\Om$,
\begin{subequations}\label{E:k ests}
\begin{gather}
  k(a,b) \ge \log\left(1+\frac{l(a,b)}{\del(a)\wedge \del(b)}\right)  
    \ge\log\left(1+\frac{|a-b|}{\del(a)\wedge \del(b)}\right) \ge\left|\log\frac{\del(a)}{\del(b)}\right|\,;  \label{E:k ge j}
  \intertext{where $l(a,b)$ is the (intrinsic) length distance between $a$ and $b$. The first inequality above is a special case of the more general (and easily proven) inequality}
  \ell_k (\gam) \ge \log\Bigl(1+\frac{\ell(\gam)}{\dist(|\gam|,\bOm)} \Bigr)  \label{E:k ge log}
\end{gather}
\end{subequations}
which holds for any rectifiable path $\gam$ in $\Om$.  See also \cite[(2.3),(2.4)]{BHK-unif}.  There are analogous inequalities for $k_\chi$ and $k_\sig$ where we replace all the Euclidean metric quantities by the appropriate chordal or spherical metric quantities.

\smallskip

It is well known that the \holo\ covering $\mfC\xra{\exp}\Cstar$ pulls back the quasihyperbolic metric $\del_\star^{-1}ds$ on $\Cstar$ to the Euclidean metric on $\mfC$, which in turn reveals that $(\Cstar,k_\star)$ is (isometric to) the Euclidean cylinder $\mfS^1\times\mfR^1$ with its Euclidean length distance inherited from the standard embedding into $\mfR^3$; here $k_\star:=k_{\Cstar}$.  See \cite{MO-qhyp}.  In particular, quasihyperbolic geodesics in $\Cstar$ are logarithmic spirals and for all $a,b\in\Cstar$,
\begin{gather}
  k_\star(a,b)=\bigl|\Log(b/a)\bigr|=\bigl|\log|b/a|+i\Arg(b/a)\bigr| \notag
  \intertext{and thus}
  \bigl|\log\frac{|b|}{|a|}\bigr|\vee\bigl|\Arg\bigl(\frac{b}{a}\bigr)\bigr| \le  k_\star(a,b) \le
  \bigl|\log\frac{|b|}{|a|}\bigr| + \bigl|\Arg\bigl(\frac{b}{a}\bigr)\bigr| \le
  \bigl|\log\frac{|b|}{|a|}\bigr| + \frac{\pi}{2}\,\frac{|a-b|}{|a|\wedge|b|}.  \label{E:k* ests}
\end{gather}
Note the special cases of the above that arise when $|a|=|b|$ or $\Arg(b/a)=0$.

Often, \eqref{E:k* ests} provides good estimates for quasihyperbolic distances as described next.

\begin{fact}  \label{F:k ests in A}%
Suppose $A=\{r<|z-c|<R\}\subset\Om$ with $c\in\mfC\sm\Om$ and $R/r>4$.  Let $\del_*^{-1}ds$ and $k_*$ denote the quasihyperbolic metric and distance in the punctured plane $\mfC\sm\{c\}$.  Then in $\{2r<|z-c|<R/2\}$, $\half\del_*\le\del\le\del_*$ and $k_*\le k\le 2k_*$.
\end{fact} %

Another consequence of \eqref{E:k* ests} is the following description for the two `ends' of $(\Cstar,k_\star)$.  We include this trivial observation as motivation for later results; see \rf{L:nmlzd punx disks} and \rf{P:QI equiv punx disks}

\begin{lma} \label{L:punx disks in Cstar} %
Let $r\in(0,+\infty)$.  Define $\Del_\star,\Delstar\subset\Cstar$by
\[
  \Del_\star:=\D[0;r]\sm\{0\}
  \quad\text{and} \quad
  \Delstar:=\mfC\sm\D(0;r)\,.
\]
Then both $(\Del_\star,k_\star)$ and $(\Delstar,k_\star)$ are $\pi$-roughly isometrically equivalent to the infinite ray $\bigl([0,+\infty),\ed\bigr)$.
\end{lma} %
\begin{proof}%
It is easy to see that the map $(\Del_\star,k_\star)\xra{\vphi}\bigl([0,+\infty),\ed\bigr)$, $\vphi(z):=\log(|z|/r)$, is a surjective $(1,\pi)$-QI equivalence; also, here $\vth(s):=r e^{-s}$ is an isometric embedding from $\bigl([0,+\infty),\ed\bigr)$ into $(\Del_\star,k_\star)$ with
\[
  \vth\bigl([0,+\infty)\bigr)=(0,r]\subset\Del_\star\subset N_{k_\star}\bigl[(0,r];\pi\bigr]\,.
\]
For $(\Delstar,k_\star)$, we use the fact that $z\mapsto z^{-1}$ is an isometric automorphism of $(\Cstar,k_\star)$.
\end{proof}%

As the chordal and spherical distances are \BL\ equivalent, and $\sig$ is the length distance associated with $\chi$ (on $\RS$), it follows that the chordal and spherical quasihyperbolic metrics (and their associated distances) are \BL\ equivalent with $\ksig\le \kchi \le \frac{\pi}2\, \ksig$.

It is useful to know that Euclidean and spherical quasihyperbolic distances are \BL\ equivalent.  Note, however, that the distortion constant depends on the location of the origin.  Essentially, this is because a general \MT\ is not a chordal nor spherical isometry, just \BL.  One can establish this by using appropriate estimates between $\del$ and $\chi$ as in \cite[Lemma~3.10]{BHX-inversion} or alternatively appeal to \cite[Theorem~4.12]{BHX-inversion}; the interested reader should also peruse \cite{BB-gromov}.

\begin{fact}  \label{F:k BL kh} %
Let $\Om\subsetneq\mfC$ be a domain.  Then $(\Om,k), (\Om,\kchi)$, and $(\Om,\ksig)$ are all \BL\ equivalent.  In particular, $\qtr k\le \kchi \le 8 (1+D)^4 k$ where $D:=\dist(0,\mfC\sm\Om)$.\footnote{One should view the constant $D$ as depending on $\diam_\chi\Om$.  By examining the distances between 0 and 1 in $\D(0;R)$ (or in $\mfC\sm\{R\}$) (as $R\to+\infty$) we see that this \BL\ constant really does depend on $D$.}
\end{fact} %

\medskip

An important property of hyperbolic distance is its conformal invariance.  While this does not hold for quasihyperbolic distance, it is \Mob\ quasi-invariant in the following sense; see \cite[Lemma~2.4, Corollary~2.5]{GP-qch}.  Modifications to their argument also gives quasi-invariance of the relative distances $j, j'$.

\begin{fact}  \label{F:kQI} %
Let $\RS\xra{T}\RS$ be a \MT.  Let $\Om\subsetneq\mfC$ and suppose $\Om':=T(\Om)\subset\mfC$.  Put $\del':=\del_{\Om'}, k':=k_{\Om'}$, and $j':=j_{\Om'}$.  Then
\begin{gather*}
  \forall\; z\in\Om\,, \quad  \frac1{\del(z)} \le2\,\frac{|T'(z)|}{\del'(T(z))}.
  \intertext{Consequently, for all rectifiable paths $\gam$ in $\Om$,}
  \tfrac12 \ell_k(\gam) \le \ell_{k'}(T\comp\gam) \le 2\ell_k(\gam).
  \intertext{In particular, with $a':=T(a)$ and $b':=T(b)$,}
  \forall\; a,b\in\Om\,, \quad \tfrac12 k(a,b) \le k'(a',b') \le 2k(a,b)
  \intertext{and also}
  \forall\; a,b\in\Om\,, \quad \tfrac12 j(a,b) \le j'(a',b') \le 2j(a,b).
\end{gather*}
\end{fact} %

In \cite[Theorem~3]{GO-unif}, Gehring and Osgood proved the following, which says that \qc\ \homeo s are rough quasihyperbolic quasiisometries and conformal maps are even quasihyperbolically \BL.


\begin{fact}  \label{F:GO-k-RQI} %
For each $K\ge1$ there is a constant $C=C(K)$ such that for any $\Om\xra{f}\Om'$, a $K$-QC \homeo\ between two proper plane domains,
\[
  \forall\; a,b\in\Om\,, \quad k'\bigl(f(a),f(b)\bigr) \le C \max\{k(a,b),k(a,b)^p\}
\]
where $p:=K^{-1}$.  In particular, a conformal map between $\Om$ and $\Om'$ induces a \BL\ equivalence between $(\Om,k)$ and $(\Om',k')$.
\end{fact} %

\subsubsection{The Hyperbolic Metric}    \label{ss:hyp metric} %
Every hyperbolic domain in $\RS$ carries a unique metric, $\lam\,ds=\lam_\Om\,ds$, which enjoys the property that its pullback $p^*[\lam\,ds]$, \wrt any holomorphic universal covering projection $p:\mfD\to\Om$, is the hyperbolic metric $\lam_\mfD(\zeta)|d\zeta|=2(1-|\zeta|^2)^{-1}|d\zeta|$ on $\mfD$.  Another description is that $\lam\,ds$ is the unique maximal (or unique complete) metric on $\Om$ that has constant Gaussian curvature $-1$.  In terms of such a covering $p$, the metric-density $\lam=\lam_\Om$ of the {\em Poincar\'e hyperbolic metric\/} $\lam_\Om\,ds$ can be determined from
$$
  \lam(z)=\lam_\Om(z)=\lam_\Om(p(\zeta))=2(1-|\zeta|^2)^{-1}|p'(\zeta)|^{-1},
$$
the above being valid for points $z\in\Om\cap\mfC$ whereas one must use local coordinates in any neighborhood of the point at infinity if $\Om\not\subset\mfC$.  (Alternatively, one can use the \emph{chordal hyperbolic metric-density} $\hat{\lam}$ and then the hyperbolic metric is $\hat{\lam}\,\hat{ds}$ where $\hat{ds}$ denotes the chordal (or spherical) arclength ``differential''.)

For example, the hyperbolic metric $\lam_*ds$ on the punctured unit disk $\mfD_*:=\mfD\sm\{0\}$ can be obtained by using the universal covering $z=\exp(w)$ from the left-half-plane onto $\mfD_*$ and we find that
\[
  \lam_*(z)|dz|=\frac{|dz|}{|z|\bigl|\log|z|\bigr|}.
\]

The \emph{hyperbolic distance} $h=h_\Om$ is the length distance $h_\Om:=d_{\lam}$ induced by the hyperbolic metric $\lam\,ds$ on $\Om$.  This is a geodesic distance: for any points $a,b$ in $\Om$, there is an $h$-geodesic $[a,b]_h$ joining $a,b$ in $\Om$.  These geodesics need not be unique, but they enjoy the property that
\[
  h(a,b)=\ell_h([a,b]_h).
\]
Here we are writing $\ell_h$ in lieu of $\ell_\lam$.

Except for a short list of special cases, the actual calculation of any given hyperbolic metric is notoriously difficult; computing hyperbolic distances and determining hyperbolic geodesics is even harder.  Indeed, one can find a number of papers analyzing the behavior of the hyperbolic metric in a twice punctured plane.  Typically one is left with estimates obtained by using domain monotonicity and considering `nice' sub-domains and super-domains in which one can calculate, or at least estimate, the metric.

The standard technique for estimating the hyperbolic metric and hyperbolic distance is via domain monotonicity, a consequence of Schwarz's Lemma.  That is, if $\Om_{\rm in}\subset\Om\subset\Om_{\rm out}$, then in $\Om_{\rm in}$, $\lam_{\rm in}ds \ge\lam\,ds\ge\lam_{\rm out}ds$ and $h_{\rm in}\ge h \ge h_{\rm out}$.

Notice that the largest hyperbolic plane regions are twice punctured planes.  We write $\lam_{ab}\,ds$ and $h_{ab}$ for the hyperbolic metric and hyperbolic distance in the twice punctured plane $\Cab$.  The standard twice punctured plane is $\Coo$ and its hyperbolic metric has been extensively studied by numerous researchers including \cite{Hempel-twice}, \cite{Minda-LNM}, \cite{SV-estimates}, \cite{SV-ineqs}.  We mention only the following.

\begin{fact}  \label{F:lam_01} %
For all $z\in\Coo$, $\lam_{01}(z)\ge\lam_{01}(-|z|)\ge\bigl( |z|[\kk+\bigl|\log|z|\bigr|] \bigr)^{-1}$, with equality at $z$ \ifff $z=-1$.
Here $\kk:= \bigl(\lam_{01}(-1)\bigr)^{-1} = \Gam^4(1/4)/\bigl(4\pi^2\bigr)=4.3768796\dots$.
\end{fact}%
\noindent
\rf{F:lam_01} was first proved by Lehto, Virtanen and \Va\ (see \cite{LVV-distortion}); later proofs were given by Agard \cite{Agard-distortion}, Jenkins \cite{Jenkins-LandauII}, and Minda \cite{Minda-LNM}. 

\smallskip

For future reference we record the following well-known estimates for hyperbolic distance in $\mfD_\star$, in $\mfD^\star$, and in $\Coo$.  Since we know the hyperbolic metrics in $\mfD_\star$ and $\mfD^\star$, it is easy to check the first two estimates.  The estimates for $h_{01}$ are straightforward consequences of \rf{F:lam_01} above.

\begin{facts}  \label{F:h in Ds&C01} %
\rule{1cm}{0mm}
\begin{enumerate}[\rm(a), wide, labelwidth=!, labelindent=0pt]
  \item  For any $a,b\in\mfD_\star$, $\ds h_\star(a,b)\le\Bigl|\log\frac{\log(1/|a|)}{\log(1/|b|)}\Bigr|+\frac{\pi}{\log2}$.
  \item  For any $a,b\in\mfD^\star$, $\ds h^\star(a,b)\le\Bigl|\log\frac{\log|a|}{\log|b|}\Bigr|+\frac{\pi}{\log2}$.
  \medskip
  \item  For all $a,b\in\Coo$:
         \begin{align*}
           &1\le|a|\le|b| \implies
             h_{01}(a,b)\ge h_{01}(-|a|,-|b|) \ge\log\frac{\kk+\log|b|}{\kk+\log|a|}
           \intertext{and}
           &|a|\le|b|\le1 \implies
             h_{01}(a,b) \ge h_{01}(-|a|,-|b|) \ge \log\frac{\kk+\log(1/|a|)}{\kk+\log(1/|b|)}\,.
         \end{align*}
\end{enumerate}
\end{facts} %


\subsubsection{The Beardon-Pommerenke Function $\bp$}  \label{ss:BP} %
We desire upper and lower estimates for the hyperbolic metric in terms of the quasihyperbolic metric.
These metrics are 2-bi-Lipschitz equivalent for simply connected hyperbolic plane regions; this is false for any domain with an isolated boundary point (such as the punctured unit disk).  The hyperbolic and quasihyperbolic metrics are bi-Lipschitz equivalent precisely when $\RS\sm\Om$ is uniformly perfect (see \cite{BP-beta}, \cite{Pomm-unifly-perfect1}, \cite{Pomm-unifly-perfect2}).  Beardon and Pommerenke corroborated this latter assertion as an application of their elegant result \cite[Theorem~1]{BP-beta} which says:
\begin{noname}
For any hyperbolic region $\Om$ in $\mfC$ and for all $z\in\Om$,
\[
  \frac{1}{\del(z)\bigl(\kk+\bp(z)\bigr)} \le \lam(z)  
  \le \frac{\pi/2}{\del(z) \, \bp(z)} .  \tag{\BPt}
\]
\end{noname}
\noindent
Here the domain function $\Om\xra{\bp}\mfR$, introduced by Beardon and Pommerenke, is defined via
\[
  \bp(z)=\bp_{\Om}(z):=
  \inf_{\substack{
    {\zeta\in\B(z)}\\
     \xi\in(\mfC\sm\Om)\sm\{\zeta\}}}
    \biggl|\log\Bigl|\frac{\zeta-z}{\zeta-\xi}\Bigr|\biggr|\,;
\]
note that the infimum is restricted to nearest boundary points $\zeta\in\B(z)=\bOm\cap\bD(z)$ for $z$ (that is, $\zeta\in\bOm$ with $\del(z)=|z-\zeta|$).  Also,
\[\kk:= \bigl(\lam_{01}(-1)\bigr)^{-1} = \Gam^4(1/4)/\bigl(4\pi^2\bigr)=4.3768796\dots.\]

The definition of $\bp$ is motivated by examining the standard lower bound for the hyperbolic metric on a twice punctured plane. The (\BPt) inequalities follow via domain monotonicity: the upper bound for $\lam(z)$ holds because $z$ lies on the conformal center of a certain annulus in $\Om$, and the lower bound holds because $\Om$ lies in a certain twice punctured plane.

A geometric interpretation for $\bp(z)$ is seen by defining $\bp(z,\zeta)$---for points $z\in\Om$ and $\zeta\in\B(z)$---as follows:
\[
  \bp(z,\zeta):=\inf_{\xi\in(\mfC\sm\Om)\sm\{\zeta\}} \biggl|\log\Bigl|\frac{\zeta-z}{\zeta-\xi}\Bigr|\biggr|.
  \qquad\text{(Thus, $\ds\bp(z)=\inf_{\zeta\in\B(z)}\bp(z,\zeta)$)}.
\]
We find that $2\bp(z,\zeta)$ is the conformal modulus of the maximal Euclidean annulus that is contained in $\Om$ and symmetric \wrt the circle $\Sone(\zeta;\del(z))$.  It follows that $2\bp(z)$ is the minimum of these numbers; so, $2\bp(z)$ is the smallest of these maximal moduli.

Thus we see that whenever $\bp(z)>0$, there is an annulus
\begin{gather*}
  \BP(z):=\A(\zeta;\del(z),\bp(z)) \in\mcA^1_\Om
  \intertext{associated with $z$; here $\zeta\in\B(z)$ is any nearest boundary point for $z$ that realizes $\bp(z)$, and}
  \forall\;\xi\in\bd\BP(z)\cap\bOm\,,\quad \bp(z)=\bp(z,\zeta)=\biggl|\log\Bigl|\frac{\zeta-z}{\zeta-\xi}\Bigr|\biggr|.
\end{gather*}
We call $\BP(z)$ a \emph{Beardon-Pommerenke annulus} (or briefly, a \emph{\BPt\ annulus}) associated with the point $z$; it needn't be unique.

\smallskip

Typically, the Beardon-Pommerenke inequalities (\BPt) are employed to give lower estimates for hyperbolic distance, but they can also provide useful upper estimates.

\begin{xx}  \label{X:h est} %
Let $0<a<b$ with $\log(b/a)>2$.  Put $\Om:=\mfC\sm\{0,-a,-b\}$.  Then
\[
  h(a,b) \le 4+\pi\log\Bigl(\half\log\frac{b}{a}\Bigr)\,.
\]
\end{xx} %
\begin{proof}%
Let $s:=\sqrt{ab}$.  Evidently, for $x\in[a,b]$, $\del(x)=x$ and
\begin{gather*}
  \bp(x)=\bigl|\log\frac{x}{a}\bigr|\wedge\bigl|\log\frac{b}{x}\bigr|\,.
  \intertext{Thus by (\BPt)}
  h(ae,s)\le\frac{\pi}2\int_{ae}^s\frac{dx}{x\log(x/a)}
    =\frac{\pi}2\log\log\bigl(\frac{b}{a}\bigr)^{1/2}
  \intertext{and similarly, $h(s,b/e)\le\frac{\pi}2\log\log(b/a)^{1/2}$.  Hence}
  h(a,b)\le 2k(a,ae)+h(ae,b/e)+2k(b/e,b) \le 4+\pi\log\Bigl(\half\log\frac{b}{a}\Bigr)
\end{gather*}
as claimed.
\end{proof}%


In \cite[Proposition~3.3]{HB-geodesics} we established especially useful estimates for $\bp$; these say that in any annulus in $\mcA^2_\Om(8\log2)$, the domain function $\bp$ decays `linearly' as we move away from the center circle.  The assumption that both boundary circles of the annulus meet $\bOm$ is crucial for obtaining the upper bounds; when only one boundary circle has a boundary point, $\bp$ can actually increase when we move away from the center circle towards the boundary circle that does not meet $\bOm$.  Here is a summary.

\begin{fact}  \label{F:BPEP} %
Let $\A(o;d,r)\in\mcA^2_\Om(8\log2)$, so $r>\log16$. For $|t|\lex r$ and all $z\in\Sone(o;de^t)$, $\bp(z)\eqx r-|t|$.  More precisely, for $|t|\le r-\log16$ and $z\in\Sone(o;de^t)$,
\[
  \frac12 \bigl( r-|t| \bigr) \le \bp(z) \le2 \bigl( r-|t| \bigr).
\]
In the above, $z\in\A[o;d,r-\log16]$ and $t:=\log(|z-o|/d)$.  In particular:

\smallskip\noindent(a)  In $\A[o;d,r-\log16]$, $\bp\le 2r=\md(A)$, and on $\Sone(A)=\Sone(o;d)$, $\bp\ge\half r$.

\smallskip\noindent(b)  On both boundary circles $\Sone(o;16^{\pm1}de^{\mp r})$, $2\log2 \le \bp \le 8\log2$.

\smallskip\noindent(c)  For $\log16<q<r$ we have the strict inequalities
\[
  \text{$\bp>\half\,q$ in $\A(o;d,r-q)$ \hspace{1em} and \hspace{1em} $\bp<2\,q$ in $\A(o;d,r-\log16)\sm\A[o;d,r-q]$}.
\]

\smallskip\noindent(d)  We also have estimates for $\bp$ in annuli $A=\A(o;d,r)\in\mcA^1_\Om$, as long as we move towards the boundary circle that has a boundary point.
\end{fact} %

Recall from above that for each $z\in\Om$ with $\bp(z)>0$ there is an associated \emph{Beardon-Pommerenke annulus}
\[
  \BP(z):=\A(\zeta;\del(z),\bp(z))\in\mcA^1_\Om \quad\text{(so, $\BP(z)\subset\Om$ and $\bd\BP(z)\cap\bOm\ne\emptyset$)}.
\]

\subsubsection{The ABC Property}  \label{ss:ABC} %
A path $\mfR\supset I\xra{\gam}\Om$ has the \emph{arcs bounce or cross property} with parameters $\mu>0$ and $\nu>0$, abbreviated as the \emph{$(\mu,\nu)$-ABC property}, \ifff for each compact subpath $\alf$ of $\gam$ and for each annulus $A:=\A(o;d,\nu)\in\mcA_\Om$ such that $\mathsf{S}^1(A)\supset\bd\alf$, we have $|\alf|\subset\A(o;d,\mu)$.

The ABC property implies that the path can cross a moderate size annulus at most once, so if the path enters deep into an annulus, it either stays there  or it crosses (once) and never returns.  In particular, if the path goes near an ``isolated island or archipelago'' of $\Om^c\cup\{\infty\}$, then it stays near it.  There are various precise statements of this phenomenon given in \cite[Lemma~3.7]{HB-geodesics}.

We utilize the following; see \cite[Lemma~3.7(a), Prop.~3.8]{HB-geodesics}.

\begin{facts}  \label{F:ABC info} %
\begin{enumerate}[\rm(a), wide, labelwidth=!, labelindent=0pt]
  \item  Suppose $\mfR\supset I\xra{\gam}\Om$ has the \emph{ABC property} with parameters $\mu\ge\nu>0$.  If $\A(o;d,\nu)\in\mcA_\Om$, then $\gam$ crosses $\A(o;d,\mu)$ at most once.
  \item  Quasihyperbolic geodesics have the $(\pi,\log2)$-ABC property,
  \item  Hyperbolic geodesics have the $(3\kk,5/2)$-ABC property.
\end{enumerate}
\end{facts} %

\subsection{Gromov Hyperbolicity and Uniformity}  \label{ss:GH&U} %

Thanks to the ground-breaking work in \cite{BHK-unif}, we know that Gromov hyperbolicity and uniformity are intimately connected.

Roughly speaking, a metric space is \emph{uniform} when points in it can be joined by paths which are not ``too long'' and which ``move away'' from the region's boundary.  More precisely, $\Om\subset\mfC$ is \emph{$C$-uniform} (for some constant $C\ge1$) provided each pair of points can be joined by a $C$-uniform arc.  Here a rectifiable arc $\gam:a\cra b$ is a \emph{$C$-uniform arc} \ifff it is both a \emph{$C$-quasiconvex arc} and a \emph{double $C$-cone arc}; these conditions mean, respectively, that
\begin{subequations}\label{E:unif}
 \begin{gather}
   \ell(\gam) \le C|a-b|  \label{E:u:qcx}
   \intertext{and}
   \forall\; z\in|\gam|\,, \quad \ell(\gam[z,a])\wedge\ell(\gam[z,b])\le C \del(z)\,.  \label{E:u:dcone}
 \end{gather}
\end{subequations}

Martio and Sarvas introduced the notion of a uniform domain in \cite{MS-unif}, and this has proven to be invaluable in geometric function theory and especially for the ``analysis in metric spaces'' program.  A simply connected proper subdomain of the plane is uniform \ifff it is a quasidisk.  Each uniform domain has the Sobolev extension property, and the BMO extension property characterizes uniformity.  See \cite{G-qdisks} and the many references therein, especially \cite{Jones-sob-ext,Jones-bmo-ext}.

\medskip

A geodesic metric space $X$ is \emph{Gromov hyperbolic} if there exists a constant $\tha\ge0$ such that every geodesic triangle is \emph{$\tha$-thin}, meaning that each point on any edge of the triangle is at distance at most $\tha$ from the other two edges.  See \cite[Chapter 3]{BHK-unif}, \cite{Brid-Haef}, \cite{BBI-metricgeometry}, or \cite{V-Gromov} and the many references in these.

The \emph{Gromov boundary} $\bd_G X$ of a Gromov hyperbolic space $X$ is the set of equivalence classes of geodesic rays, where two rays are equivalent \ifff their Hausdorff distance is finite.  One can also use quasi-geodesic rays, or, Gromov sequences.  There is no canonical preferred distance on the Gromov boundary.  However, for each $\veps\in(0,\veps_0]$ (usually $\veps_0=\veps_0(\tha):=1\wedge(1/5\tha)$) there is a so-called \emph{visual distance} $d_\veps=d_{\veps,o}$ on $\bd_G X$ that satisfies
\begin{gather*}
  \half\, \exp\bigl(-\veps(\xi\vert\eta)_o\bigr) \le d_\veps(\xi,\eta) \le \exp\bigl(-\veps(\xi\vert\eta)_o\bigr)
  \intertext{for all $\xi,\eta\in\bd_G X$, where $(\xi\vert\eta)_o$ is the usual Gromov product and $o\in X$ a fixed base point.  Standard estimates then give}
  C^{-1} \exp\Bigl(-\veps\, \dist(o,(\xi,\eta)\bigr) \Bigr) \le d_\veps(\xi,\eta) \le C \exp\Bigl(-\veps\,
  \dist(o,(\xi,\eta)\bigr) \Bigr)
\end{gather*}
where $(\xi,\eta)$ is any geodesic line in $X$ with endpoints $\xi,\eta\in\bd_G X$.

The \emph{conformal gauge on $\bd_G X$} is the maximal collection of all distance functions on $\bd_G X$ that are \qsc ally equivalent to some (hence all) visual distance(s).

%
\section{Proofs of Theorems}  \label{S:Proofs} 

\flag{Should we indicate a proof about when $(\Om,h)$ and $(\Om,k)$ are isometric?  We could create an Appendix where we put this and perhaps other items e.g. related to uniform perfectness and a proof of \rf{L:cfml metrics}.}

\subsection{Proof of Theorem~\ref{TT:QSequiv}}  \label{s:pfThmQSequiv} %

We employ the following technical fact about ``fat'' annuli.

\begin{lma} \label{L:fat ann} %
Suppose $m:=\half\md(A)>1$ for some annulus $A\subset\Om$ with center in $\mfC\sm\Om$ and $\bA\cap\bOm\ne\emptyset$.
Then there are points $a,b,c\in A$ with:
\addtocounter{equation}{-1}  
\begin{subequations}\label{E:chi ests}
  \begin{alignat}{2}
    &\half(m-1)\le k \le m-1 \,, &\quad&\text{for $k\in\{k(a,b),k(c,b)\}$},              \label{E:ks comp}
    \intertext{and}
    &h(a,b)\ge\log\Bigl(1+\frac{m-1}{2(\kk+1)}\Bigr)\,, &&\text{whereas $h(c,b)\le1.1$}. \label{E:hs diff}
  \end{alignat}
\end{subequations}
\end{lma} 
\begin{proof}%
By similarity invariance we may assume $A=\{z:e^{-m}<|z|<e^m\}$ and either $-e^{-m}\in\bOm$ or $-e^m\in\bOm$; here $m>1$.  Assume $-e^{-m}\in\bOm$ and define
\[
  a:=e^{1-m}\,,\quad b:=\sqrt{a}\,,\quad  c:=1.
\]
We demonstrate that these points possess the asserted properties.

Thanks to \rf{F:k ests in A} we know that $k_\star\le k \le 2k_\star$ in $A$.  Since
\[
  2\,k_\star(a,b)=2\,\log\frac{|b|}{|a|}=\log\frac1{|a|}=m-1=2\,k_\star(c,b),
\]
the inequalities in \eqref{E:ks comp} follow.

Since $\Om\subset\mfC\sm\{0,-e^{-m}\}$, an appeal to \rf{F:h in Ds&C01}(c) provides the estimate
\begin{align*}
    h(a,b) &\ge h_{-e^{-m}0}(a,b)=h_{01}(-e^m a,-e^m b)  \\
           &\ge\log\frac{\kk+\log(e^m b)}{\kk+\log(e^m a)} = \log\Bigl(1+\frac{m-1}{2(\kk+1)}\Bigr).
\end{align*}

Employing (\BPt) we deduce that
\begin{align*}
  h(c,b) &\le \int_{[b,c]}\lam\,ds \le\frac{\pi}2\,\int_{[b,c]}\frac{ds}{\del\,\bp}  \\
         &=\frac{\pi}2\,\int_b^c \frac{dt}{t(m+\log t)} = \frac{\pi}2\,\log\frac{2m}{m+1} < 1.1.
\end{align*}

When $-e^{m}\in\bOm$ we take $a:=e^{m-1},b:=\sqrt{a},c:=1$ and argue similarly.
\end{proof}%

\begin{pf}{Proof of \rf{TT:QSequiv}} %
Thanks to the work \cite{BP-beta} of Beardon and Pommerenke, it suffices to explain why (\ref{TT:QSequiv}.1) implies \eqref{E:bounded mod}, but perhaps it is illuminating to see why (\ref{TT:QSequiv}.2) implies \eqref{E:bounded mod}; these implications are both quantitative.

With this in mind, suppose $(\Om,h)\xra{f}(\Om,k)$ is $K$-\BL.  Let $A\in\mcA_\Om$, and let $\alf$ be the simple loop in $\Om$
whose trajectory is the center circle $|\alf|=\Sone(A)$.  Using hyperbolic distance in $A\subset\Om$, we deduce that
\begin{gather*}
  \ell_h(\alf)\le2\pi^2/\md(A).
  \intertext{Since $f\comp\alf$ is an essential loop in $\Om$, we obtain}
  2\pi\le\ell_k(f\comp\alf)\le K\,\ell_h(\alf) \le 2K\pi^2/\md(A)
\end{gather*}
and thus $\md(A)\le\pi K$.

\medskip

Now suppose $(\Om,h)\xra{f}(\Om,k)$ is $\eta$-QS.  Thanks to \rf{L:cfml metrics} we know that $\Om\xra{f}\Om$ is $K$-QC where
$K=K(\eta(1))$ depends only on the value $\eta(1)$.  Let $C=C(K)=C(\eta(1))$ be the constant from \rf{F:GO-k-RQI}.  We establish
\eqref{E:bounded mod} with the upper bound
\[
  M:=\max\{4C+2,2+4(\kk+1)\exp\Bigl(1.1\bigl(\eta^{-1}(1/2C^2)\bigr)^{-1}\Bigr)\}.
\]

Let $A\in\mcA_\Om$.  By enlarging $A$ if necessary, we may assume that $\bA\cap\bOm\ne\emptyset$, and also that
$m:=\half\md(A)>2C+1$.  Let $a,b,c$ be the points in $A$ given by \rf{L:fat ann} and let $a',b',c'$ be their $f$ images.  Since
$k(a,b)\ge\half(m-1)\ge1$, \rf{F:GO-k-RQI} tells us that
\begin{gather*}
  k(a',b')\le C\,k(a,b).
  \intertext{The same fact, now applied to $f^{-1}$, tells us that}
  C\,\max\{k'(c',b'),k'(c',b')^p\}\ge k(c,b) \ge \half(m-1) \ge C,
  \intertext{so $k'(c',b')\ge 1$ whence $k(c,b)\le C\,k'(c',b')$.  Therefore}
  k'(a',b')\le C\,k(a,b) \le 2C\, k(c,b) \le 2C^2 k'(c',b')
  \intertext{and thus by \qsy}
  \frac1{2C^2} \le \frac{k'('c',b')}{k'(a',b')} \le \eta\Bigl(\frac{h(c,b)}{h(a,b)}\Bigr)
                \le \eta\Bigl(1.1\bigl(\log(1+\frac{m-1}{2(\kk+1)}\bigr)^{-1} \Bigr)
\end{gather*}
which gives the asserted estimate $2m\le M$.
\end{pf} %

\subsection{Proof of Theorem~\ref{TT:Gromov}}  \label{s:pfThmGromov} %

It seems plausible, especially in light of \rf{P:For unif=>QS}(g) below (and its quasihyperbolic analog), that one could give a direct proof of \rf{TT:Gromov}.  The authors are unable to do so, and we instead base our proof on the following hyperbolic analog of \cite[Theorems~3.6]{BHK-unif}.  Note that it depends heavily on \cite[Theorems~1.11, 1.12, and Prop.~7.12]{BHK-unif} and these in turn depend on the Bonk-Heinonen-Koskela uniformization theory.  Our proof also utilizes the fact \cite[Theorem~A]{HB-geodesics} that hyperbolic and quasihyperbolic quasi-geodesics are the same curves; in particular, in uniform domains  hyperbolic geodesics are uniform arcs \cite[Remarks~4.3]{DAH-hlexj}.

\begin{thm} \label{T:unif=>QS} %
Let $\Om$ be a hyperbolic domain in $\RS$.  Suppose $(\Om,\lsig)$ is uniform.  Then the canonical conformal gauge on $\bd_G(\Om,h)$ is naturally \qsc ally equivalent to the conformal gauge on $\bd(\Om,\lsig)$
determined by $\lsig$.\footnote{%
Recall that here $\lsig$ denotes the intrinsic length distance in $(\Om,\sig)$.}
\end{thm} 
The above is quantitative, but the constants are somewhat murky!

Bonk, Heinonen, and Koskela established a similar result \cite[Theorem~3.6]{BHK-unif} for abstract uniform metric spaces but
using quasihyperbolic distance in lieu of hyperbolic distance.  We closely follow their proof, but there are significant
modifications that we detail.

In particular, we utilize the following information; much of this is either a direct consequence of work in \cite{BHK-unif}, or
follows by similar reasoning, the latter being especially true whenever only upper estimates for quasihyperbolic distance are
employed (because always, $h\le2k$).  See especially \cite[Chapters~2 and 3]{BHK-unif}.  We sketch the ideas.

\flag{Surely with Ferrand distance we can seriously improve the constant mess.}

Below, and later, when $(\Om,h)$ is Gromov hyperbolic (in which case $(\Om,k)$ is also Gromov hyperbolic), we write $h_\veps=h_{\veps,o}$ and $k_\veps=k_{\veps,o}$ for the standard visual distances on $\bd_G(\Om,h)$ and $\bd_G(\Om,k)$ repectively; here, as in \rf{ss:GH&U}, the visual parameter $\veps\in(0,\veps_0]$ and $o\in\Om$ is a fixed base point.

\begin{prop} \label{P:For unif=>QS} %
Let $\Om$ be a hyperbolic domain in $\RS$ with $(\Om,\lsig)$ $A$-uniform.  There are constants $\tha,B,C$ (that depend only on the ``data'') such that the following hold.
\begin{enumerate}[\rm(a), wide, labelwidth=!, labelindent=0pt]
  \item  The metric space $(\Om,h)$ is Gromov $\tha$-hyperbolic.
  \smallskip
  \item  There is $o\in\Om$ with $\sig(o)=\max_\Om\sig$, $\ds 2\le\diam(\Om,\lsig)/\sig(o)\le2A$, so $\diam(\Om,\lsig)\le2\pi A$.
  \smallskip
  \item Each pair of distinct points in $\overline{(\Om,\lsig)}$ can be joined by a hyperbolic geodesic which is a $B$-uniform arc in $(\Om,\lsig)$; when one or both points lie in $\bd(\Om,\lsig)$, we get a hyperbolic geodesic ray or line respectively.
  \smallskip
  \item  There is a natural bijection between $\bd_G(\Om,h)$ and $\bd(\Om,\lsig)$ given by $[\gam]\mapsto\gam(\infty)$ where $[\gam]\in\bd_G(\Om,h)$ is the equivalence class of a hyperbolic geodesic ray $\gam$ that has the endpoint $\gam(\infty)\in\bd(\Om,\lsig)$.  Therefore, we use the same notation for points in $\bd_G(\Om,h)$ or $\bd(\Om,\lsig)$.
  \smallskip
  \item  Given $\zeta\in\bd(\Om,\lsig)$ and a hyperbolic geodesic ray $[o,\zeta)_h$ in $\Om$, each $\xi\in\bd(\Om,\lsig)$ has an associated point $x=x(\xi)\in[o,\zeta)_h$ with
      \begin{gather*}
        h\bigl(x,(\xi,\zeta)_h\bigr)\le C
        \intertext{and}
        h(o,x)-C \le  h\bigl(o,(\xi,\zeta)_h\bigr) \le h(o,x)+C.
        \intertext{In fact, if $\gam$ is the $\sig$-arclength parametrization for $[o,\zeta)_h$, starting at $\gam(0)=\zeta$,
        then we can take}
        x:=\begin{cases}
          \gam\bigl(\lsig(\xi,\zeta)\bigr) &\text{when $\lsig(\xi,\zeta)\le\half\lsig(o,\zeta)$} \\
          o &\text{otherwise}.
        \end{cases}
      \end{gather*}
  \smallskip
  \item  Given $\zeta,\xi\in\bd(\Om,\lsig)$ and $x=x(\xi)\in[o,\zeta)_h$ as above, we have
      \[
        C^{-1} e^{-\veps h(x,o)} \le h_\veps(\zeta,\xi) \le C e^{-\veps h(x,o)}.
      \]
\end{enumerate}
\end{prop} 
\begin{proof}%
There is no harm in rotating the sphere $\RS$, so we can assume that $\Om\subset\mfC$.
To see (a), we start with the fact (see \cite[Theorem~3.6]{BHK-unif}) that $(\Om,\ksig)$ is $\tha$-hyperbolic with $\tha=\tha(A)$, so by \rf{F:k BL kh} $(\Om,k)$ is $\tha$-hyperbolic with $\tha=\tha(A,\dist(0,\mfC\sm\Om))$.  Then
\cite[Theorem~B]{HB-geodesics} tells us that $(\Om,h)$ is also $\tha$-hyperbolic with $\tha=\tha(A,\dist(0,\mfC\sm\Om))$.

Part (b) is elementary.  For (c), we note that by \cite[Theorem~A]{HB-geodesics}, geodesic segments in $(\Om,h)$ are quasi-geodesics in $(\Om,k)$ (with an absolute constant) and hence by \rf{F:k BL kh} are also quasi-geodesics in $(\Om,\ksig)$ now with a constant that depends on $\dist(0,\mfC\sm\Om)$.  Finally, \cite[Theorem~4.1]{DAH-hlexj} affirms that these arcs are $B$-uniform with $B=B(A,d)$.  The assertions about geodesics rays and lines that end at boundary points follow in standard ways as explained in \cite[Proposition~3.12]{BHK-unif}.

Item (e)
can be established exactly as done in \cite[Lemma~3.14]{BHK-unif} for  quasihyperbolic distance.  Evidently, (f) follows from (e) and the standard estimates for visual distances given at the end of \rf{ss:GH&U}.

Item (d) follows mostly as in \cite[Proposition~3.12]{BHK-unif} with one major modification.  It is routine to see that each hyperbolic geodesic ray in $\Om$ has an endpoint in $\bd(\Om,\lsig)$, that rays with the same endpoint are equivalent, and that each boundary point is the endpoint of such a ray.  It remains to show that equivalent rays have the same endpoint.  Suppose $\alf$ and $\beta$ are hyperbolic geodesic rays in $\Om$ with $\xi:=\alf(\infty),\eta:=\beta(\infty)\in\bd(\Om,\lsig)$, and $\xi\ne\eta$.  We claim that $\dist_{\mcH}^h(|\alf|,|\beta|)=+\infty$ (so $\alf$ and $\beta$ are not equivalent).

This is not difficult to check when $\xi$ and $\eta$ correspond to\footnote{%
The identity map $(\Om,\lsig)\xra{\id}(\Om,\sig)$ is $1$-Lipschitz, so has a $1$-Lipschitz extension to a map $\overline{(\Om,\lsig)}\xra{\iota}(\hat{\Om},\sig)$ and the $\iota$ image of $\bd(\Om,\lsig)$ is precisely the set of rectifiably accessible boundary points of $(\Om,\sig)$; see \cite[Prop.~3.22]{DAH-diam-dist}.)
\label{fn:iota}}
different points in $\hat\bd\Om=\bd(\Om,\sig)$, but requires additional effort if these two length boundary points are attached to the \emph{same} spherical  boundary point, which we assume is the origin $0$.  Since we are ``near'' the origin, we can work with Euclidean quantities in place of spherical.

As in \cite[Proposition~3.12]{BHK-unif}, since $\xi\ne\eta$, there is a quasihyperbolic geodesic line $\gam=(\xi,\eta)_k$ in $\Om$.  Let $\Lam:=\ell(\gam)$ and let $z_o$ be the arclength midpoint of $\gam$.  Then
\[
  |z_o|\ge\del(z_o)\ge\frac{\Lam}{2B}\,.
\]
Put $L:=\min\{\Lam/10B,\ell(\alf),\ell(\beta)\}$.  Pick $s_a<s_o<s_b$ so that $z_o=\gam(s_o)$ and and so the quasihyperbolic subrays $\gam_a(s):=\gam(s_a-s), \gam_b(s):=\gam(s_b+s)$ (for $s\in[0,+\infty)$) of $\gam$ both have length
\[
  \ell(\gam_a)=L=\ell(\gam_b).
\]
By trimming the initial parts of $\alf,\beta$ (if necessary), we may assume they both have length $L$.  Let $a:=\gam(s_a),b:=\gam(s_b)$ be the initial points of $\gam_a,\gam_b$ respectively.

Note that $\gam_a,\gam_b$ are quasihyperbolic subrays of $\gam$ with $\gam_a(\infty)=\xi,\gam_b(\infty)=\eta$.  As $\gam$ is also a hyperbolic quasi-geodesic line in $\Om$ (by \cite[Theorem~A]{HB-geodesics}) and $\alf(\infty)=\gam_a(\infty),\beta(\infty)=\gam_b(\infty)$, there is a finite constant $H$ such that
\[
  \forall\;s\ge0\,, \quad h\bigl(\alf(s),\gam_a(s)\bigr) \le H \quad\text{and}\quad h\bigl(\beta(s),\gam_b(s)\bigr)\le H.
\]
\flag{give $H$?}

Since $\xi\ne\eta$, there is ``plenty'' of $\bOm$ ``near'' the origin.  In particular, it is not hard to check that $\bp\le\log10$ on $|\gam_a|\cup|\gam_b|$, so by (BP) $\lam\,ds$ and $\del^{-1}ds$ are \BL\ on $|\gam_a|\cup|\gam_b|$.  It follows that for all $s>|s_a|\vee|s_b|$,
\begin{gather*}
  \ell_h(\gam[-s,s_a]) \eqx \ell_k(\gam[-s,s_a]) = k\bigl(\gam(-s),a\bigr)
  \intertext{and}
  \ell_h(\gam[s_b,s]) \eqx \ell_k(\gam[s_b,s]) = k\bigl(\gam(s),b\bigr),
\end{gather*}
so $\ell_h(\gam[-s,s]) \gex k(\gam(-s),a) + k(\gam(s),b)$.
Finally, for all sufficiently large $s>0$,
\begin{align*}
  h\bigl(\alf(s),\beta(s)\bigr) &\ge h\bigl(\gam(-s),\gam(s)\bigr)-h\bigl(\alf(s),\gam(-s)\bigr)-h\bigl(\beta(s),\gam(s)\bigr) \\
  &\ge h\bigl(\gam(-s),\gam(s)\bigr) -2H  \gex \ell_h(\gam[-s,s]) -2H \\
  &\gex k\bigl(\gam(-s),a\bigr) + k\bigl(\gam(s),b\bigr) -2H  \to +\infty \;\text{(as $s\to\infty$)}.
\end{align*}
Thus $\alf$ and $\beta$ are indeed non-equivalent hyperbolic geodesic rays :-)
\end{proof}%

We require the following technical information.  The upshot of this is that, given two length boundary points, we can always find spherical boundary points at a distance comparable to the length distance between the two given length boundary points.  Here $\iota$ is as described in footnote \ref{fn:iota}. Also, we employ the ABC property for hyperbolic geodesics; see \rf{ss:ABC}.

\begin{lma} \label{L:Euc bdy pts} %
Let $\cc :=e^{-(\mu +\nu )}$ where $\mu ,\nu $ are ABC parameters for hyperbolic distance.\footnote{%
We take $\mu:=3\kk$ and $\nu:=5/2$; then $\mu>\nu$ and $\cc$ is an absolute constant.}
Let $\Om\subset\RS$ be a hyperbolic domain.  Suppose $\xi,\eta$ are distinct points in $\bd(\Om,\lsig)$.  Assume there is a hyperbolic geodesic line $\gam:=(\xi,\eta)_h$ that is also a $B$-uniform arc in $(\Om,\lsig)$.  Let $z_0$ be the $\sig$-arclength midpoint of $\gam$.  Put
$\xi_0:=\iota(\xi), \eta_0:=\iota(\eta)$, and $r_0:=\chi(z_0,\eta_0)$.  Then
\addtocounter{equation}{-1}  
\begin{subequations}\label{E:Euc bdy pts}
  \begin{gather}
    \text{either \; $\chi(\xi_0,\eta_0)\ge r_0$\,, \; or} \quad  \hat{\bd}\Om\cap\Set{z\in\RS| \cc r_0\le \chi(z,\eta_0)\le
    r_0}\ne\emptyset.            \label{E:A:Euc bdy pts}
    \intertext{Thus there exists a point $\xi_1\in\bOmh$ such that $\chi(\xi_1,\eta_0)\eqx\lsig(\xi,\eta)$; more
    precisely,}
    \frac{\cc}{\pi B}\,\lsig(\xi,\eta) \le \chi(\xi_1,\eta_0) \le \frac{B}2\, \lsig(\xi,\eta).   \label{E:B:Euc bdy pts}
  \end{gather}
\end{subequations}
\end{lma} %
\begin{proof}%
We will see, after some normalization, that \eqref{E:A:Euc bdy pts} follows directly from the ABC property and then
\eqref{E:B:Euc bdy pts} is an easy consequence of uniformity.  To verify \eqref{E:A:Euc bdy pts}, assume $\chi(\xi_0,\eta_0)<\cc
r_0$.  By rotating $\RS$ if necessary, we may assume that $\eta_0=0\in\mfC$.

Note that
\begin{gather*}
  f(t):=\frac{t}{\sqrt{4-t^2}} \quad\text{has $f(\chi(z,0))=\abs{z}$},
  \intertext{so $f$ is increasing on $[0,2)$; also,}
  f^{-1}(s)=\frac{2s}{\sqrt{1+s^2}}.
  \intertext{Thus $f(r_0)=f(\chi(z_0,\eta_0))=\abs{z_0}$ and for any $r\in(0,r_0)$,}
  A:=\{z\mid r < \chi(z,\eta_0) < r_0\} = \{z\mid f(r) < \abs{z} < f(r_0)\}\,.
\end{gather*}

We set $r:=f^{-1}(\cc |z_0|)$, so: $f(r)=\cc \abs{z_0}, A=\{\cc|z_0|<|z|<|z_0|\}$, and  $\md(A)=\mu+\nu$.  Also, note that
\[
  r= f^{-1}(\cc |z_0|) = \frac{2\cc |z_0|}{\sqrt{1+(\cc |z_0|)^2}} > \frac{2\cc |z_0|}{\sqrt{1+|z_0|^2}} = \cc  r_0 >
  \chi(\xi_0,\eta_0)=f^{-1}(|\xi_0|)\,.
\]
So $|\xi_0|<\cc |z_0|$.  This gives $A\cap\bOm\ne\emptyset$, as we explain below, and then \eqref{E:A:Euc bdy pts} follows.

To this end, notice that as $|\xi_0|<\cc |z_0|$, there is a subarc $\alf$ of $\gam$ with $z_0\in|\alf|$ and $\bd\alf\subset\Sone(0;e^{-\mu }|z_0|)$.  Since $\mu>\nu$, $\A(0;e^{-\mu}|z_0|,\nu)\subset A$.  If $A\subset\Om$ were true, then by employing the fact that $\gam$ enjoys the $(\mu,\nu)$-ABC property (see \rf{ss:ABC}) we could assert that $|\alf|\subset\A(0;e^{-\mu},\mu)=A$; but, $z_0\in|\alf|$ and $z_0\notin A$. Therefore, $A\cap\bOm\ne\emptyset$, so \eqref{E:A:Euc bdy pts} holds.

\smallskip

Now we establish \eqref{E:B:Euc bdy pts}.  As $\gam$ is a $B$-uniform arc in $(\Om,\lsig)$,
\begin{gather*}
  \lsig(\xi,\eta) \le \ells(\gam) \le B \lsig(\xi,\eta).
  \intertext{Also,}
  r_0=\chi(z_0,\eta_0) \le \ells(\gam[z_0,\eta]) = \half\,\ells(\gam) \le \frac{B}2\, \lsig(\xi,\eta)
  \intertext{and}
  r_0\ge\chi(z_0)\ge\frac2{\pi}\,\sig(z_0)\ge\frac2{\pi B}\,\ells(\gam[z_0,\eta])\ge \frac1{\pi B}\,\lsig(\xi,\eta).
  \intertext{Thus if $\chi(\xi_0,\eta_0)\ge r_0$, then}
  \frac1{\pi B}\,\lsig(\xi,\eta) \le r_0 \le \chi(\xi_0,\eta_0) \le \sig(\xi_0,\eta_0) \le \lsig(\xi,\eta)
  \intertext{and \eqref{E:B:Euc bdy pts} holds with $\xi_1:=\xi_0$.  Suppose $\chi(\xi_0,\eta_0)< r_0$.  Then by \eqref{E:A:Euc
  bdy pts} there is a point $\xi_1\in\bOmh$ with}
  \frac{\cc}{\pi B}\, \lsig(\xi,\eta) \le cr_0 \le \chi(\xi_1,\eta_0) \le r_0 \le \frac{B}2\, \lsig(\xi,\eta)
\end{gather*}
as asserted in \eqref{E:B:Euc bdy pts}.
\end{proof}%

Armed with the notation and results from \rf{P:For unif=>QS} and \rf{L:Euc bdy pts}, we now establish \rf{T:unif=>QS}.

\begin{pf}{Proof of \rf{T:unif=>QS}} %
We may assume $\Om\subset\mfC$.  Then $(\Om,\ksig)\xra{\id}(\Om,k)$ is \BL\ as explained in \rf{F:k BL kh}.  Let $\cc:=e^{-(\mu+\nu)}$ be the constant in \rf{L:Euc bdy pts}.

As in \cite[Theorem~3.6]{BHK-unif}, we show that the bijection $\bd(\Om,\lsig)\to\bd_G(\Om,h)$ (given in \rf{P:For unif=>QS}(d)) is a \qsy; here we assume $h_\veps$ is a standard visual distance on $\bd_G(\Om,h)$ as in \rf{ss:GH&U} with visual parameter $\veps\in(0,\veps_0]$ and $o\in\Om$ is a fixed base point as given in \rf{P:For unif=>QS}(b).

Let $\zeta,\eta,\xi$ be points in $\bd(\Om,\lsig)$ and put $t:=\lsig(\zeta,\xi)/\lsig(\zeta,\eta)$.  When $t\ge1$, we can copy the Bonk-Heinonen-Koskela argument as it only uses upper estimates for quasihyperbolic distances.\footnote{%
Here we use the fact that $h\le2k\le8\ksig$.}
Thus we may, and do, assume that $t<1$.

Let $x=x(\xi),y=y(\eta)$ be the points on the hyperbolic geodesic ray $[o,\zeta)_h$ that are given by \rf{P:For unif=>QS}(e)
and associated with $\xi,\eta$ respectively.  Since $t<1$, we have $\zeta<x\le y\le o$ where the geodesic is ordered from $\zeta$ to $o$.  Then from \rf{P:For unif=>QS}(f) we find that
\begin{gather*}
  \frac{h_\veps(\zeta,\xi)}{h_\veps(\zeta,\eta)} \le C\, e^{-\veps h(x,y)} \le C.
  \intertext{It follows that for any fixed $t_0\in(0,1)$,}
  \forall\;t\in[t_0,1)\,,\quad
  \frac{h_\veps(\zeta,\xi)}{h_\veps(\zeta,\eta)} \le \frac{C}{t_0}\, t.
  \intertext{To finish the proof, we demonstrate that for all $0<t<t_0:=\cc/\pi B^2$, $h(x,y)\ge H(t)$ where $H(t)\to+\infty$ as $t\to0^{+}$.  This then gives}
  \frac{h_\veps(\zeta,\xi)}{h_\veps(\zeta,\eta)} \le C e^{-\veps H(t)}\to0\;\text{as $t\to0^{+}$}
\end{gather*}
which in turn confirms that the bijection $\bd(\Om,\lsig)\to\bd_G(\Om,h)$ is indeed \qsc.

Before immersing ourselves in the details, we explain the idea.  Whenever one knows three distinct boundary points, one has a standard lower bound for hyperbolic distance given by looking at the appropriate thrice punctured sphere; \rf{F:h in Ds&C01}(c) is handy for estimating hyperbolic distance in such a domain.  A difficulty here is that we have points in $\bd(\Om,\lsig)$ whereas we need points in $\hat{\bd}\Om$.  To overcome this, we appeal to \rf{L:Euc bdy pts}.

With this in mind, let $\zeta_0:=\iota(\zeta),\xi_0:=\iota(\xi), \eta_0:=\iota(\eta)$; see footnote \ref{fn:iota}.  By \rf{L:Euc bdy pts} there are points $\xi_1,\eta_1\in\bOmh$ with $\xi_1\ne\eta_1$ (when $t<t_0$) and such that
\begin{subequations}\label{E:Euc bdy pts in Pf}
\begin{gather}
  \frac{\cc}{\pi B}\,\lsig(\xi,\zeta) \le \chi(\xi_1,\zeta_0) \le \frac{B}2\, \lsig(\xi,\zeta)  \label{E:A:Euc bdy pts in Pf}
  \intertext{and}
  \frac{\cc}{\pi B}\,\lsig(\eta,\zeta) \le \chi(\eta_1,\zeta_0) \le \frac{B}2\, \lsig(\eta,\zeta).  \label{E:B:Euc bdy pts in Pf}
\end{gather}
\end{subequations}
Let $x':=T(x),y':=T(y)$ denote the images of $x,y$ (respectively) under the \MT\ $T$ that maps $\zeta_0,\xi_1,\eta_1$ to
$0,1,\infty$ respectively; so
\[
  T(z):=[z,\zeta_0,\eta_1,\xi_1]=\frac{(z-\zeta_0)(\eta_1-\xi_1)}{(z-\eta_1)(\zeta_0-\xi_1)}.
\]

Since $\Om\subset\Om_0:=\RS\sm\{\zeta_0,\xi_1,\eta_1\}$, writing $h_0:=h_{\Om_0}$, we now have
\begin{gather*}
  h(x,y)\ge h_0(x,y)=h_{01}(x',y')\ge h_{01}(-|x'|,-|y'|)\,.
  \intertext{To complete the proof, we demonstrate below that}
  \text{as}\;t\to0^+\,,\quad |x'|\eqx 1 \quad\text{and}\quad  |y'|\eqx\frac1t
\end{gather*}
which, in conjunction with \rf{F:h in Ds&C01}(c), provides the desired estimate.

First we show that $1\eqx\ds|x'|=|x,\zeta_0,\eta_1,\xi_1|=\frac{\chi(x,\zeta_0)\,\chi(\eta_1,\xi_1)}{\chi(x,\eta_1)\,\chi(\zeta_0,\xi_1)}$.  We
claim that
\begin{gather*}
  \frac{4}{\pi B^2}\, \chi(\xi_1,\zeta_0) \le \chi(x,\zeta_0) \le \frac{\pi B}{\cc}\, \chi(\xi_1,\zeta_0)\,,
    \quad\text{so}\quad
    \frac{\chi(x,\zeta_0)}{\chi(\zeta_0,\xi_1)} \le \frac{\pi B}{\cc}.
  \intertext{To see this, we use the definition of $x:=\gam\bigl(\lsig(\xi,\zeta)\bigr)$ along with \eqref{E:A:Euc bdy pts in Pf} to obtain}
  \chi(x,\zeta_0)\le\sig(x,\zeta_0)\le\ells(\gam[x,\zeta])=\lsig(\xi,\zeta)\le\frac{\pi B}{\cc}\,\chi(\xi_1,\zeta_0);
  \intertext{and also, as $\zeta_0\in\bOmh$ and $\gam$ is $B$-uniform in $(\Om,\lsig)$,}
  \frac{\pi}{2}\,\chi(x,\zeta_0) \ge \sig(x,\zeta_0) \ge \sig(x)
   \ge B^{-1}\ells(\gam[x,\zeta])=B^{-1}\lsig(\xi,\zeta) \ge
  \frac{2}{B^2}\,\chi(\xi_1,\zeta_0).
\end{gather*}

Next, continuing to use the inequalities \eqref{E:Euc bdy pts in Pf} we deduce that
\begin{gather*}
  \chi(\xi_1,\eta_1) \le \chi(\zeta_0,\xi_1)+\chi(\zeta_0,\eta_1) \le \frac{B}2\bigl(\lsig(\zeta,\xi)+\lsig(\zeta,\eta)\bigr) = \frac{B}2 (1+t) \lsig(\zeta,\eta) \le B\, \lsig(\zeta,\eta)
  \intertext{and, when $t\le t_0=\cc/\pi B^2$,}
  \chi(\xi_1,\eta_1) \ge \chi(\zeta_0,\eta_1)-\chi(\zeta_0,\xi_1) \ge \frac{\cc}{\pi B}\lsig(\zeta,\eta) -\frac{B}2\lsig(\zeta,\xi)
     = \Bigl(\frac{\cc}{\pi B} - \frac{B}2t\Bigr) \lsig(\zeta,\eta) \ge \frac{\cc}{2\pi B} \lsig(\zeta,\eta).
  \intertext{From the above inequalities we see that, when $t\le t_0$,}
  \frac{\cc}{2\pi B}\, \lsig(\zeta,\eta) \le \chi(\xi_1,\eta_1) \le B\, \lsig(\zeta,\eta).
\end{gather*}

Replacing $\xi_1$ with $x$ and repeating the argument directly above, we check that when $t\le t_1:=\cc/2\pi B$,
\begin{gather*}
  \frac{\cc}{2\pi B}\, \lsig(\zeta,\eta) \le \chi(x,\eta_1) \le B\, \lsig(\zeta,\eta).
  \intertext{Combining the above we now find that when $t\le t_0<t_1$,}
  \frac{\chi(\xi_1,\eta_1)}{\chi(x,\eta_1)} \le \frac{B\,\lsig(\zeta,\eta)}{(\cc/2\pi B)\lsig(\zeta,\eta)} = \frac{2\pi}{\cc}\,B^2.
  \intertext{Finally,}
  |x'| = \frac{\chi(x,\zeta_0)\,\chi(\eta_1,\xi_1)}{\chi(x,\eta_1)\,\chi(\zeta_0,\xi_1)}
       \le \frac{\pi B}{\cc}\,\frac{2\pi}{\cc}\,B^2 = \frac{2\pi^2}{\cc^2} B^3.
\end{gather*}

It remains to explain why $\ds\frac1t\eqx|y'|=\frac{\chi(y,\zeta_0)\,\chi(\eta_1,\xi_1)}{\chi(y,\eta_1)\,\chi(\zeta_0,\xi_1)}$.  From above, we already know that
\begin{gather*}
  \frac{\cc}{\pi B^2}\,\frac1t = \frac{(\cc/2\pi B)\lsig(\zeta,\eta)}{(B/2)\lsig(\xi,\zeta)}  \le
    \frac{\chi(\eta_1,\xi_1)}{\chi(\zeta_0,\xi_1)} \le \frac{B\, \lsig(\zeta,\eta)}{(\cc/\pi B)\lsig(\xi,\zeta)}
    = \frac{\pi B^2}{\cc}\, \frac1t
  \intertext{when $t\le t_0$.  Thus it suffices to demonstrate that $\chi(y,\zeta_0)\eqx\chi(y,\eta_1)$.  There are two cases depending on whether or not $y=o$.  When $y=o$ we find that}
  \frac1{\pi A} \le \frac{\chi(y,\zeta_0)}{\chi(y,\eta_1)} \le \pi A\,;
  \intertext{the diligent reader can confirm this with the help of \rf{P:For unif=>QS}(b).  We assume $y\ne o$, or equivalently, $\lsig(\zeta,\eta)\le\half\lsig(\eta,o)$.  Here $y=\gam\bigl(\lsig(\eta,\zeta)\bigr)$, so}
  \chi(y,\zeta_0) \le \lsig(y,\zeta) \le \ells(\gam[y,\zeta])= \lsig(\eta,\zeta) \quad\text{and}\quad
  \chi(y,\eta_1) \le \chi(y,\zeta_0)+\chi(\eta_1,\zeta_0)\le2\lsig(\eta,\zeta)
  \intertext{and then by uniformity}
  \chi(y,\zeta_0) \wedge \chi(y,\eta_1) \ge \chi(y) \ge \frac2{\pi}\,\sig(y) \ge \frac2{\pi B} \lsig(\eta,\zeta)
  \intertext{whence}
  \frac1{\pi B} \le \frac{\chi(y,\zeta_0)}{\chi(y,\eta_1)} \le \frac{\pi}2\,B.
\end{gather*}
\end{pf} %

\begin{pf}{Proof of \rf{TT:Gromov}} %
Since Euclidean translations are isometries of both $(\Om,h)$ and $(\Om,k)$, there is no harm in assuming that the origin lies in $\bOm$.  In this setting, $(\Om,k)$ and $(\Om,\ksig)$ are \BL\ equivalent with an absolute constant.  It now follows that $(\Om,h),(\Om,k)$ and $(\Om,\ksig)$ are all Gromov hyperbolic (or not, but we assume the former).

It is well known (see \cite[Theorem~1, p.211]{Goluzin-GFT} or \cite[Theorem~IX.22]{Tsuji-PotlThy}) that $\Om$ is conformally equivalent to a \emph{horizontal slit domain} $\Om'\subsetneq\RS$; thus $\infty\in\Om'$ and each boundary component of $\Om'$ is either a point or a compact horizontal line segment.  One can show that $(\Om',\lsig')$ is $C$-LLC$_2$ with $C:=\pi/\sig'(\infty)$.  Therefore by \cite[Prop.7.12]{BHK-unif}, $(\Om',\lsig)$ is uniform.

We are now positioned to apply \rf{T:unif=>QS} and its quasihyperbolic analog \cite[Theorem~3.6]{BHK-unif}.  These two results provide the first two QS equivalences $\cong$ below:
\[
  \bd_G(\Om,h) \equiv \bd_G(\Om',h') \cong \bd(\Om,\lsig) \cong \bd_G(\Om',\ksig') \cong \bd_G(\Om,k);
\]
the isometric equivalence $\equiv$ holds because conformal maps are hyperbolic isometries, and the last QS equivalence $\cong$ holds because $(\Om',\ksig')$ and $(\Om,k)$ are \BL\ equivalent thanks to \rf{F:GO-k-RQI}.
\end{pf} %


\subsection{Proof of Theorem~\ref{TT:isolated}}  \label{s:pfThmIsolated} %
Below, in \rf{ss:pf UP+S}, we establish the following general result; this provides a large class of plane domains whose hyperbolizations and quasihyperbolizations are \qic ally equivalent.  In particular, each finitely connected domain belongs to this class, thus corroborating \rf{TT:isolated}.  However, this class also includes many infinitely connected domains such as $\mfC\sm\mfZ$ or $\Cstar\sm\{\frac1n\mid n\in\mfN\}$.

\begin{thm} \label{T:UP+S} %
Let $\Pi\subset\Om\subset\RS$ be as described in \rf{ss:NKA} with \eqref{E:ass4UP+S} holding.  In addition, suppose that
\[
  \RS\sm\Om_\Del = (\RS\sm\Om)\cup\Del \quad\text{is $M$-uniformly perfect.}
\]
Then $(\Om_\Pi,h_\Pi)$ and $(\Om_\Pi,k_\Pi)$ are $(L,6\pi)$-\qic ally equivalent where $L=L(M)$.
\end{thm} 

\flag{It would be nice to have useful, easily verifiable, conditions that guarantee that $\RS\sm\Om_\Del$ is uniformly perfect!  For example, assume $\Om\subsetneq\mfC$ and $r_p:=\half\del(p)$.}

\subsubsection{Notation and Key Assumptions}    \label{ss:NKA} %
Let $\Pi\subset\Om\subset\RS$ and put $\Om_\Pi:=\Om\sm\Pi$.  Here $\Om$ can be $\mfC$ or even $\RS$, but we assume that $\Pi$ is closed in $\Om$, that $\Om_\Pi\subset\mfC$, and that \emph{$\Om_\Pi$ is hyperbolic}; in particular, the point at infinity belongs to $\Om$ \ifff it belongs to $\Pi$.  For convenience, we set $\Pi^\star:=\Pi\sm\{\infty\}$.

We assume that for each $p\in\Pi$ there is an associated $r_p>0$ with the properties described below.  For $p\ne\infty$, we set
\[
  \Del_p:=\D[p;r_p] \quad\text{and}\quad \Delstar_p:=\Del_p\sm\{p\}\,,
\]
and we assume that
\begin{subequations}\label{E:ass4UP+S}
 \begin{align}
   &2 r_p \le \del(p)
     \quad\text{(so $\D(p;2r_p)\subset\Om$)}
     \label{E:1:a4UPS}
   \intertext{and that for points $p\ne q$ in $\Pi^\star$,}
   &2(r_p+r_q) \le |p-q|
     \quad\text{(so $\D(p;2r_p)\cap\D(q;2r_q)=\emptyset$).}  \label{E:2:a4UPS}
   \intertext{To handle finitely connected domains, e.g., $\mfD^\star$ or $\Coo$, we must allow $\infty\in\Pi$ in which case we further assume that}
   &\RS\sm\Om \subset \D[0;\tfrac14 r_\infty]
     \quad\text{and that}\quad \Del_p\subset\D[0;\tfrac14 r_\infty]
       \quad\text{for each $p\in\Pi^\star$}\,.
     \label{E:3:a4UPS}
 \end{align}
\end{subequations}
When $\infty\in\Pi$, we set
\begin{gather*}
  \Del_\infty:=\RS\sm\D(0;r_\infty) \quad\text{and}\quad \Delstar_\infty:=\Del_\infty\sm\{\infty\}\,.
  \intertext{Next, define}
  \Del:=\bigcup_{p\in\Pi}\Del_p \quad\text{and}\quad \Om_\Del:=\Om\sm\Del\,.
\end{gather*}
In the above, $\del=\del_\Om$ is the Euclidean distance to $\bOm$ which is infinite if $\Om\supset\mfC$.

The hypotheses above ensure that $\Pi$ is discrete in $\Om$, so $\Om_\Pi$ is a domain.  Similarly, $\Del$ is closed in $\Om$ and $\Om_\Del$ is a domain.  We use the subscripts $\Pi$ and $\Del$ to denote quantities associated with $\Om_\Pi$ and $\Om_\Del$.  For example, $h_\Pi=h_{\Om_\Pi}$ and $k_\Pi=k_{\Om_\Pi}$ are the hyperbolic and quasihyperbolic distances in $\Om_\Pi$ (respectively), and especially $\bp_\Pi=\bp_{\Om_\Pi}$.

\subsubsection{Quasiisometric Equivalence of Punctured Disks}    \label{ss:QIEpxD} %
Here we present some technical details that allow us to streamline our proof of \rf{T:UP+S} and focus on the underlying ideas.  Roughly, we show that whenever $\Del_*$ is a punctured disk in a hyperbolic plane domain $\Om$ (with the puncture a point of $\RS\sm\Om$), then $(\Del_*,h)$ and $(\Del_*,k)$ are \qic ally equivalent.

We use the following numerical result whose proof is left for the diligent reader.
\begin{equation}\label{E:QIEpxD}
  \forall\; K>0,L>0,x>0,y\ge L\,: \quad
         \frac{K+x}{K+y}\ge\frac{L}{K+L}\frac{x}{y}\,.
\end{equation}

We require a result similar to \rf{L:punx disks in Cstar} but for punctured disks in arbitrary domains; here there are separate cases for finite versus infinite punctures.  Since Euclidean similarity transformations are quasihyperbolic isometries, we can normalize as in the following.

\flag{We don't use second part, but it's good motivation for proof of \rf{P:QI equiv punx disks}!}

\begin{lma} \label{L:nmlzd punx disks} %
Suppose $\mfD\subset\Om\subset\mfC$ with $1\in\mfC\sm\Om$.  Let $\Om_*:=\Om\sm\{0\}$ and $k_*:=k_{\Om_*}$.  For each $r\in(0,\half]$, $(\Del_*,k_*)$ is $\pi$-roughly isometrically equivalent to the infinite ray $\bigl([0,+\infty),\ed\bigr)$, where $\Del_*:=\D[0;r]\sm\{0\}$.

Next, suppose $0\in\RS\sm\Om\subset\cmfD$.  Let $\Om^*:=\Om\sm\{\infty\}$ and $k^*:=k_{\Om^*}$.  For each $R\ge2$, $(\Del^*,k^*)$ is $(2,\pi)$-\qic ally equivalent to the infinite ray $\bigl([0,+\infty),\ed\bigr)$, where  $\Del^{\!*}:=\mfC\sm\D(0;R)$.
\end{lma} %
\begin{proof}%
In $\{0<|z|<\half\}\supset\Del_*$, $\del_*(z)=|z|=\del_\star(z)$, so the identity map
\[
  (\Om_*,k_*)\supset(\{0<|z|<1/2\},k_*)\xra{\id}(\{0<|z|<1/2\},k_\star) \subset(\Cstar,k_\star)
\]
is an isometric equivalence and thus \rf{L:punx disks in Cstar} gives the first assertion.
The second assertion holds because $z\mapsto z^{-1}$ is quasihyperbolically $2$-\BL, as per \rf{F:kQI}.  Alternatively, it is not hard to check that for $|z|\ge2$, $\half|z|\le\del^*(z)\le|z|=\del_\star(z)$.
\end{proof}%

We conclude this subsubsection with a technical result that we employ in our proof of \rf{T:UP+S}.  As above there are separate cases for finite versus infinite punctures.

\begin{prop} \label{P:QI equiv punx disks} %
Suppose $\mfD\subset\Om\subset\mfC$ with $1\in\mfC\sm\Om$.  Assume $\Om_*:=\Om\sm\{0\}$ is hyperbolic and let $h_*:=h_{\Om_*},k_*:=k_{\Om_*}$.  For each $r\in(0,\half]$, $(\Del_*,h_*)$ is $C_1$-roughly isometrically equivalent to $(\Del_*,k_*)$, where $\Del_*:=\D[0;r]\sm\{0\}$ and $C_1:=\pi/\log2\le5$.

Next, suppose $1\in\RS\sm\Om\subset\cmfD$.  Assume $\Om^*:=\Om\sm\{\infty\}$ is hyperbolic and let $h^*:=h_{\Om^*}, k^*:=k_{\Om^*}$.  For each $R\ge4$, $(\Del^{\!*},h^*)$ is $(2,2\pi)$-\qic ally equivalent to $(\Del^{\!*},k^*)$, where  $\Del^{\!*}:=\mfC\sm\D(0;R)$.
\end{prop} %
\begin{proof}%
First, suppose $\mfD\subset\Om\subset\mfC$ with $1\in\mfC\sm\Om$ and $\Om_*=\Om\sm\{0\}$ hyperbolic.  Fix $r\in(0,\half]$ and let $\Del_*:=\D[0;r]\sm\{0\}$.  As in \rf{L:nmlzd punx disks},
\[
  \bigl([0,+\infty),\ed\bigr)\xra{\vth}(\Del_*,k_*)\subset(\Om_*,k_*)\,, \quad \vth(s):=re^{-s}
\]
is an isometric embedding and a $\pi$-rough isometric equivalence.

We show that the map
\begin{gather*}
  (\Del_*,h_*)\xra{\psi}\bigl([0,+\infty),\ed\bigr)\quad\text{given by}\quad  \psi(z):=\log\frac{\log(1/|z|)}{\log(1/r)}
  \intertext{is a surjective rough isometric equivalence.  The inclusions $\Del_*\subset\mfD_\star\subset\Om_*\subset\Coo$ tell us that}
  \text{in $\mfD_\star$}\,,\quad h_\star\ge h_*\ge h_{01}
  \quad\text{(here $h_\star=h_{\mfD_\star}$ and $h_{01}=h_{\Coo}$)}\,. 
  \intertext{Employing the estimates in \rfs{F:h in Ds&C01} we deduce that for all $0<|a|\le|b|\le\half$:}
  h_*(a,b) \le h_\star(a,b) \le \log\frac{\log(1/|a|)}{\log(1/|b|)} + \frac{\pi}{\log2}
  \intertext{and}
  h_*(a,b) \ge h_{01}(a,b) \ge \log\frac{\kk+\log(1/|a|)}{\kk+\log(1/|b|)} \ge
  \log\biggl(\frac{\log2}{\kk+\log}\frac{\log(1/|a|)}{\log(1/|b|)}\biggr)\,;
  \intertext{here \eqref{E:QIEpxD} provides the last inequality.  Thus for all $a,b\in\Del_*$,}
  h_*(a,b)-\frac{\pi}{\log2} \le \abs{\psi(a)-\psi(b)} \le h_*(a,b) + \log\Bigl(1+\frac{\kk}{\log2}\Bigr)
\end{gather*}
and so $\psi$ is indeed a surjective $(1,C_1)$-QI equivalence.

It is now not difficult to confirm that the map
\[
  (\Om_*,h_*)\supset(\Del_*,h_*)\xra{\Phi:=\vth\comp\psi}(\Del_*,k_*)\subset(\Om_*,k_*)\,,
  \quad\Phi(z):=\vth\bigl(\psi(z)\bigr)=re^{-\psi(z)}=\dfrac{r\log(1/r)}{\log(1/|z|)}
\]
is a $C_1$-rough isometric equivalence.

\medskip

Next, we examine an infinite puncture.  Suppose $1\in\RS\sm\Om\subset\cmfD$ with $\Om^*=\Om\sm\{\infty\}$ hyperbolic.  Fix $R\ge4$ and let $\Del^{\!*}:=\mfC\sm\D(0;R)$.  Here we cannot argue as in the second part of the proof of \rf{L:nmlzd punx disks} because we do not know whether the origin lies in $\Om$ or in $\mfC\sm\Om$.  Also, unlike the first part above, where we had the three boundary points $0,1,\infty$, here we only know two boundary points.  Nonetheless, both $(\Del^{\!*},h^*)$ and $(\Del^{\!*},k^*)$ are QI equivalent to the infinite ray $\bigl([0,+\infty),\ed\bigr)$ as we now corroborate.

Put $R_1:=R+1$, $\Del_1:=\RS\sm\D(1;R_1)$, and $\Del_1^{\!*}:=\Del_1\sm\{\infty\}$.  One can appeal to \rf{L:nmlzd punx disks} with $\Del_1^{\!*}$, but a direct approach yields better estimates.
It is not hard to check that for $|z-1|\ge4$,
\begin{gather*}
  \tfrac12 |z-1| \le \del^*(z) \le |z-1|\,.
  \intertext{Arguing as in the alternative proof for the second assertion in \rf{L:nmlzd punx disks}, we deduce that the map}
  \bigl([0,+\infty),\ed\bigr)\xra{\vth}(\Del^{\!*},k^*) \quad\text{given by}\quad \vth(s):=1-R_1 e^s
  \intertext{satisfies}
  s,t\in[0,+\infty) \implies  |s-t|\le k^*\bigl(\vth(s),\vth(t)\bigr) \le 2|s-t|
  \intertext{with}
  \vth\bigl([0,+\infty)\bigr) = (-\infty,-R] \subset \Del^{\!
  *} \subset N_{k^*}\bigl[(-\infty,-R];2\pi\bigr]\,.
\end{gather*}
It follows that $\vth$ is a $(2,2\pi)$-QI equivalence.

\smallskip

Pick any $\xi\in\RS\sm\Om$ with $|\xi-1|=\max_{\zeta\in\RS\sm\Om}|\zeta-1|$;\footnote{%
Note that $|\xi-1|\ge\half\diam\bigl(\RS\sm\Om\bigr)>0$ because $\Om^*$ is hyperbolic.}
so $|\xi|\le1$, $0<|\xi-1|\le2$, and $\RS\sm\Om\subset\D[1;|\xi-1|]$.

Let $\Del':=T(\Del_1^{\!*})$, $\Om':=T(\Om^*)$, and $h':=h_{\Om'}$ where $(\Om^*,h^*)\xra{T}(\Om',h')$ is the isometric equivalence given by
\begin{gather*}
  T(z):=\frac{z-1}{\xi-1}\,.
  \intertext{Since $T\bigl(\D[1;|\xi-1|]\bigr)=\cmfD$ and $1,\xi\in\RS\sm\Om\subset\D[1;|\xi-1|]$, we find that}
  0,1\in\RS\sm\Om'\subset\cmfD\,, \quad\text{whence}\quad \mfD^\star:=\mfC\sm\cmfD\subset\Om'\subset\Coo
\end{gather*}
and therefore in $\mfD^\star$, $h^\star:=h_{\mfD^\star}\ge h'\ge h_{01}$.

In particular, for all points $a',b'\in\Del'$ (with $|b'|\ge|a'|\ge R_2:=R_1/|\xi-1|$) we have
\begin{gather*}
  h'(a',b') \le h^\star(a',b') \le \log\frac{\log|b'|}{\log|a'|} + \frac{\pi}{\log2}
  \intertext{and}
  h'(a',b') \ge h_{01}(a',b') \ge \log\frac{\kk+\log|b'|}{\kk+\log|a'|} \ge
  \log\biggl(\frac{L}{\kk+L}\frac{\log|b'|}{\log|a'|}\biggr)\,;
  \intertext{here $L:=\log R_2\ge\log(5/2)>0.9$, \eqref{E:QIEpxD} provides the last inequality, and we have used \rf{F:h in Ds&C01}(c).  Thus for all $a',b'\in\Del'$,}
  h'(a',b')-\frac{\pi}{\log2} \le \abs{\psi(a')-\psi(b')} \le h'(a',b') + \log\Bigl(1+\frac{\kk}{L}\Bigr)
  \intertext{where}
  (\Del',h')\xra{\psi}\bigl([0,+\infty),\ed\bigr) \quad\text{is given by}\quad  \psi(w):=\log\frac{\log|w|}{L}
\end{gather*}
and $\psi$ is a surjective $(1,C_1)$-QI equivalence.

Next, we employ \rf{X:QIEqX} to demonstrate that $(\Del^{\!*},h^*)$ is roughly isometrically equivalent to $(\Del_1^{\!*},h^*)$ (and hence to $(\Del',h')$).  We claim that for each $z\in\Del^{\!*}$ there is a $z_1\in\Del_1^{\!*}$ with $h^*(z,z_1)\le\half$.

Let $z_0\in\Del^{\!*}$.  If $z_0\in\Del_1^{\!*}$, put $z_1:=z_0$; assume $z_0\notin\Del_1^{\!*}$.  Let $z_1$ be the radial projection, from the origin, of $z_0$ onto $\bd\Del_1^{\!*}$.  Note that $|z_0-z_1|\le2$.  Also: $[z_0,z_1]\subset[(R/|z_0|)z_0,z_1]$ and for each $z\in[(R/|z_0|)z_0,z_1]$, $|z|\ge R\ge4$; so
\begin{gather*}
  \lam^\star(z)=\frac1{|z|\log|z|} \le \frac1{|z|} \le \frac1R \le \qtr\,.
  \intertext{In $\mfD^\star\subset\Om^*$, $\lam^\star\ge\lam^*$, and thus}
  h^*(z_0,z_1)\le h^\star(z_0,z_1)=\int_{[z_0,z_1]} \lam^\star\,ds \le \max_{[z_0,z_1]}\lam^\star\cdot |z_0-z_1| \le \half\,.
\end{gather*}
From \rf{X:QIEqX}, the map $(\Del^{\!*},h^*)\xra{f}(\Del_1^{\!*},h^*)$, $f(z):=z_1$ (with $z_1
$ as given in the above claim), is a $1$-rough isometric equivalence.

\bigskip

Finally, we now have
\begin{gather*}
  (\Del^{\!*},h^*)\xra{f}(\Del_1^{\!*},h^*)\xra{T}(\Del',h')\xra{\psi}\bigl([0,+\infty),\ed\bigr)\xra{\vth}(\Del^{\!*},k^*)
  \intertext{and a careful review of our estimates reveals that}
  \Psi:=\psi\comp T\comp f\,,  \quad\Psi(z)=\log\frac{\log|w_1|}{\log R_2} \;\text{with}\; w_1:=T(z_1)\,,
  \intertext{is a rough isometric equivalence.  Moreover, $\Phi:=\vth\comp\Psi$ is $(2,6)$-QI with}
  \Phi(\Del^{\!*})=\vth\bigl([0,+\infty)\bigl)=(-\infty,-R]\subset\Del^{\!*}\subset N_{k^*}\bigl[(-\infty,-R];2\pi\bigr]\,,
\end{gather*}
and thus $\Phi$ provides the asserted QI equivalence between $(\Del^{\!*},h^*)$ and $(\Del^{\!*},k^*)$.
\end{proof}%

\begin{cor} \label{C:PunxDisks} %
Let $\Pi\subset\Om\subset\RS$ be as described in \rf{ss:NKA}.  Then for each $p\in\Pi$ there is an $(L,C)$-\qic\ equivalence
\[
  (\Om_\Pi,h_\Pi)\supset(\Del_p,h_p)\xra{\Phi_p}(\Del_p,k_p) \subset(\Om_\Pi,k_\Pi)
\]
where $h_p$ and $k_p$ denote the distances $h_\Pi$ and $k_\Pi$ restricted to $\Del_p$ (respectively).  We can take $(L,C)=(1,\pi/\log2)$ unless $p=\infty$ in which case $(L,C)=(2,2\pi)$.
\end{cor} %
\begin{proof}%
First, assume $p\in\Pi\sm\{\infty\}$.  Let $\Om_p:=\Om_\Pi\cup\{p\}$ and pick a point $\xi\in\bd\Om_p$ nearest to $p$.  Put $T(z):=(z-p)/(\xi-p)$.  Then $\Om':=T(\Om_p)$ satisfies the hypotheses of the first part of \rf{P:QI equiv punx disks}.  Also, $r'_p:=r_p/|\xi-p|\le\half$, so there is a $\pi/\log2$-rough isometric equivalence $\Phi$ from $(\Del'_*,h')$ to $(\Del'_*,k')$, where $\Del'_*=\D(0;r'_p)\sm\{0\}=T(\Del_p)\subset\Om'$ (and $h',k'$ are respectively the hyperbolic, quasihyperbolic distances in $\Om'_*=\Om'\sm\{0\}=T(\Om_\Pi)$).  It now follows that
\[
  \Phi_p:=T^{-1}\comp\Phi\comp T\,, \quad \Phi_p(z)=p+r_p \frac{\xi-p}{|\xi-p|}
      \frac{\log\bigl(r_p/|\xi-p|\bigr)}{\log\bigl(|z-p|/|\xi-p|\bigr)}\,,
\]
is a $\pi/\log2$-rough isometric equivalence from $(\Del_p,h_p)$ to $(\Del_p,k_p)$.

Next, assume $p=\infty\in\Pi$.  Put $\Om_\infty:=\Om_\Pi\cup\{\infty\}$.  Pick a point $\xi\in\RS\sm\Om_\infty$ with $|\xi|=\max_{\zeta\in\RS\sm\Om_\infty}|\zeta|$.  Let $T(z):=z/\xi$.
Then $\Om':=T(\Om_\infty)$ satisfies the hypotheses of the second part of \rf{P:QI equiv punx disks}.
Also, $R':=r_\infty/|\xi|\ge4$, so there is a $(2,2\pi)$-quasiisometric equivalence $\Phi$ from $(\Del'^*,h')$ to $(\Del'^*,k')$, where $\Del'^*=\mfC\sm\D(0;R')=T(\Delstar_\infty)\subset\Om'$ (and $h',k'$ are respectively the hyperbolic, quasihyperbolic distances in $\Om'^*=\Om'\sm\{\infty\}=T(\Om_\Pi)$).
It now follows that
\[
  \Phi_\infty:=T^{-1}\comp\Phi\comp T\,, \quad \Phi_\infty(z)=\xi\, \Phi\bigl(z/\xi\bigr)\,,
\]
is a $(2,2\pi)$-quasiisometric equivalence from $(\Delstar_\infty,h_\infty)$ to $(\Delstar_\infty,k_\infty)$.
\end{proof}%

\subsubsection{Proof of \rf{T:UP+S}}    \label{ss:pf UP+S} %
The reader is encouraged to review \rf{ss:NKA}.  In particular, note that $\RS\sm\Om_\Del$ uniformly perfect  tells us that each isolated point of $\hat{\bd}\Om$ is a limit point of $\Pi$.

Define $(\Om_\Pi,h_\Pi)\xra{\Phi}(\Om_\Pi,k_\Pi)$ by
\[
  \Phi(z):=\begin{cases}
    \Phi_p(z) \quad&\text{if $z\in\Delstar_p$}\,, \\
    z &\text{if $z\in\Om_\Del$}\,;
  \end{cases}
\]
here $(\Delstar_p,h_\Pi)\xra{\Phi_p}(\Delstar_p,k_\Pi)$ are the $(2,2\pi)$-QI equivalences given by \rf{C:PunxDisks}.  Below we demonstrate that the identity map $(\Om_\Del,h_\Pi)\to(\Om_\Del,k_\Pi)$ is $K$-\BL\ with $K=K(M)$.  An elementary, albeit tedious, argument then reveals that $\Phi$ is $2\pi$-roughly surjective and
\[
  \forall\;a,b\in\Om_\Pi\,,\quad
  \half h_\Pi(a,b)-6\pi \le k_\Pi\bigl(\Phi(a),\Phi(b)\bigr) \le K h(a,b)+6\pi\,;
\]
so $\Phi$ is a $(K,6\pi)$-\qic\ equivalence; we assume $K\ge2$.

\medskip

To begin, we demonstrate that each annulus $A\in\mcA_\Pi$ (so, $A\subset\Om_\Pi$ and $c(A)\in\mfC\sm\Om_\Pi$) with $\Sone(A)\cap\Del=\emptyset$ has modulus $\md(A)\le2M+\log4$.  Let $A:=\Set{z\in\mfC|r<|z-c|<R}$ be such an annulus.  Assume $R/r>4$; so, $\frac32 r<\sqrt{rR}<\frac12R$.  Set
\begin{gather*}
  \Piin:=\Pi\cap\Ain\,, \quad \Piout:=\Pi\cap(\Aout\cup\{\infty\})\,, \quad \Piout^\star:=\Piout\sm\{\infty\}\,.
\intertext{Define $A':=\Set{z\in\mfC|r'<|z-c|<R'}$  where}
\begin{align*}
    r'&:=\begin{cases}
    r \quad&\text{if $\Piin=\emptyset$, or, $\Piin=\{c\}$ and $\Ain\cap(\mfC\sm\Om)\neq\emptyset$}\,, \\
    \frac32 r &\text{if $\Piin\sm\{c\}\ne\emptyset$ (i.e., $c\notin\Piin\ne\emptyset$ or $\{c\}\subsetneq\Piin$)}\,, \\
    \sqrt{rR} &\text{if $\Piin=\{c\}$ and $\Ain\cap(\mfC\sm\Om)=\emptyset$}\,.
  \end{cases}
\intertext{and}
  R'&:=\begin{cases}
    R \quad&\text{if $\Piout=\emptyset$, or, $\Piout=\{\infty\}$ and $A\cap\Del_\infty=\emptyset$}\,, \\
    \half R &\text{if $\Piout^*\ne\emptyset$ (i.e., $\{\infty\}\ne\Piout\ne\emptyset$)}\,, \\
    \sqrt{rR} &\text{otherwise (i.e., $\Piout=\{\infty\}$ and $A\cap\Del_\infty\ne\emptyset$)}\,.
  \end{cases}
\end{align*}
\end{gather*}

We show below that when $R'=\sqrt{rR}$, $\Aout\cap(\mfC\sm\Om_\Pi)=\emptyset$.
Since $\Om_\Pi$ is hyperbolic, $\Aout\cap(\mfC\sm\Om_\Pi)=\emptyset=\Ain\cap(\mfC\sm\Om)$ cannot both hold (as these would imply $\RS\sm\Om_\Pi=\{c,\infty\}$).  Thus $r\le r'< R'\le R$, $A'\csubann A$ is a concentric subannulus of $A$, and it is not difficult to check that $\log(R/r)\le2\log(R'/r')+\log4$.  We claim that $A'\cap\Del=\emptyset$.  Thus $A'\subset\Om_\Del$; since $\RS\sm\Om_\Del$ is $M$-uniformly perfect, $\md(A')\le M$ and $\md(A)\le2M+\log4$ as asserted.

\medskip

Now we check that $A'\cap\Del=\emptyset$.  Employing either \eqref{E:1:a4UPS} if $c\in\bOm$ or \eqref{E:2:a4UPS} if $c\in\Pi^\star$ we find that
\begin{equation}\label{E:2rple}
  \forall\; p\in\Pi^\star\sm\{c\}\,, \quad 2r_p\le|p-c|\,.
\end{equation}

First we show that for all $q\in\Piin$, $\Del_q\subset\D[c;r']$.  Assume $\Piin\ne\emptyset$.  If $c\notin\Piin$, then \eqref{E:2rple} provides the estimate $r_q\le\half r$ for $q\in\Piin$, so $\Del_q\subset\D[c;\frac32 r]$ as asserted.  A similar argument works for the case $\{c\}\subsetneq\Piin$.  Assume $\{c\}=\Piin$.  Suppose $\Ain\cap(\mfC\sm\Om)\ne\emptyset$, and let $\zeta$ be a point in this set.  By \eqref{E:1:a4UPS}, $2r_c\le|c-\zeta|\le r$, so $\Del_c\subset\D[c;r]$ as asserted.  When $\Ain\cap(\mfC\sm\Om)=\emptyset$, we have no estimates for $r_c$, but $\Sone(A)\cap\Del=\emptyset$ means that $\Del_c\subset\D[c;\sqrt{rR}]$.

\smallskip

Next we show that for all $q\in\Piout$, $\Del_q\cap\D(c;R')=\emptyset$. Assume $\Piout\ne\emptyset$.  Suppose $q\in\Piout^\star$.  If $r_q\le\half R$, then $|q-c|-r_q\ge\half R$, and if $r_q\ge\half R$, then by \eqref{E:2rple}, $|q-c|-r_q\ge r_q\ge\half R$; thus, in both cases $\Del_q\cap\D(c;\half R)=\emptyset$.  Also, if $\infty\in\Pi$, then $|c|+R\le|c|+|q-c|\le\frac34 r_\infty$, so here $\Del_\infty\cap\D(c;R)=\emptyset$.

It remains to examine the case $\Piout=\{\infty\}$ and $A\cap\Del_\infty\ne\emptyset$.  Here we have no estimates for $r_\infty$, but $\Sone(A)\cap\Del=\emptyset$ means that $\Del_\infty\cap\D[c;\sqrt{rR}]=\emptyset$.  Also, in this setting we must have $\Aout\cap(\mfC\sm\Om_\Pi)=\emptyset$.  (If there were a point $\zeta\in\Aout\cap(\mfC\sm\Om_\Pi)$, then by  \eqref{E:3:a4UPS}, $R\le|\zeta-c|\le|\zeta|+|c|\le\half r_\infty$, which would give $A\subset\D[0;\frac34 r_\infty]$ contradicting $A\cap\Del_\infty\ne\emptyset$.)

\medskip

We can use the above to verify that $\bp_\Pi\le M+2$ in $\Om_\Del$ which tells us that $\lam_\Pi\,ds$ and $\del_\Pi^{-1}\,ds$ are $(\kk+M+2)$-\BL\ equivalent in $\Om_\Del$.  However, if $\gam:a\cra b$ is either a hyperbolic or a quasihyperbolic geodesic in $\Om_\Pi$ with endpoints $a,b\in\Om_\Del$, then $\gam$ may leave $\Om_\Del$.  Nonetheless, the ABC property (see \rf{ss:ABC}) ensures that $\gam$ cannot enter too deep into $\Del\sm\Pi=\Om_\Pi\sm\Om_\Del$.  In particular, if $\gam$ enters some $\Del_p$, then by \rf{F:ABC info}(a), $|\gam|\cap\D[p;e^{-6\kk}r_p]=\emptyset$ when $p\ne\infty$ and $|\gam|\subset\D(0;e^{6\kk}r_\infty)$ when $p=\infty$.

Our final task is to corroborate that $\bp_\Pi\lex1$ in $\Om_{\tilde\Del}:=\Om\sm{\tilde\Del}$ where $\tilde\Del:=\bigcup_{p\in\Pi}\tilde\Del_p$ and
\[
  \tilde\Del_p:=\begin{cases}
                  \D[p;e^{-6\kk}r_p] \quad&\text{if $p\in\Pi^\star$}\,,  \\
                  \RS\sm\D(0;e^{6\kk}r_\infty) &\text{if $p=\infty\in\Pi$}\,.
  \end{cases}
\]
To this end, let $a\in\Om_{\tilde\Del}$ and pick $c\in\bd\Om_\Pi$ with
\[
  A:=\BP_\Pi(a)=\A(c,d;m)=\Set{de^{-m}<|z-c|<de^m}=\Set{r<|z-c|<R}\in\mcA_\Pi
\]
where $d:=\del_\Pi(a)=|a-c|$ and $\bp_\Pi(a)=m=\half\log(R/r)$.  Recall \eqref{E:2rple} and the notation $\Piin,\Piout$.  We assume $\Sone(A)\cap\Del\ne\emptyset$, so $\Ups:=\Set{p\in\Pi| \Sone(A)\cap\Del_p\ne\emptyset}\ne\emptyset$.

\smallskip

We consider several cases.  If $p\in\Ups\cap\Piin\sm\{c\}$, then by \eqref{E:2rple}
\begin{gather*}
  |p-c|\le r \quad\text{and}\quad d\le|p-c|+r_p \le \frac32|p-c|\,,
  \intertext{so}
  \sqrt{R/r}=\frac{d}{r}\le\frac32 \quad\text{whence}\quad \frac{R}{r}\le\frac94<3\,.
  \intertext{If $p\in\Ups\cap\Piout^\star$, then again by \eqref{E:2rple}}
  |p-c|\ge R \quad\text{and}\quad  d+r_p\ge|p-c|\,, \quad\text{so}\quad d\ge\half|p-c|
  \intertext{and therefore}
  \sqrt{R/r}=\frac{R}{d}\le2 \quad\text{and}\quad \frac{R}{r}\le4\,.
\end{gather*}
Thus in these two easy cases we have $\bp_\Pi(a)=\half\md(A)\le\log2$.

\smallskip

It remains to deal with the case $\Ups\cap\{c,\infty\}\ne\emptyset$; here $\Sone(A)\cap\Del_q=\emptyset$ for all $q\in\Pi^\star\sm\{c\}$.  Roughly speaking, we exhibit a concentric subannulus $A'\csubann A$ with $\md A' \eqx \md A$ and $\Sone(A')\cap\Del=\emptyset$.  It then follows from earlier work that $\md A\lex M$.

Suppose $c\in\Ups$.  Then $\Ups=\Set{c}$, $r_c\ge d$, and $R\ge2r_c$ (because $c$ is an isolated point of $\bOm_\Pi$ nearest to $a$, so $\bdo A\cap\bOm_\Pi\ne\emptyset$).  If $R\le2r_c e^{m/2}$, then as $a\notin\tilde\Del_p$,
\begin{gather*}
  2r_c \ge R e^{-m/2}=d e^{m/2} = |a-c| e^{m/2} \ge e^{-6\kk} r_c e^{m/2}
  \intertext{whence}
  \bp(a)=m\le 12\kk+\log4\,.
  \intertext{Assume $R\ge2r_c e^{m/2}$.  We claim that $r':=(2r_c)^2/R\in[r,d]$.  Therefore}
  A':=\Set{z\in\mfC|r'<|z-c|<R}\csubann A \quad\text{is a concentric subannulus of $A$}
  \intertext{with $\Sone(A')=\Sone(c;2r_c)$.  From \eqref{E:ass4UP+S}, $\Sone(A')\cap\Del=\emptyset$, so by earlier work,}
  \bp(a)=\half\md(A)=\log\frac{R}{d}\le\log\frac{R}{r'}=\md(A')\le 2M+\log4\,.
  \intertext{To check the claim, note that}
  r=\frac{d^2}R\le\frac{r_c^2}R\le r'=\frac{(2r_c)^2}R \le \frac{(Re^{-m/2})^2}R=R e^{-m}=d\,.
\end{gather*}

Suppose $\infty\in\Ups$.  Here $\Ups=\Set{\infty}$, $\frac34 r_\infty\le d\le C r_\infty$ where $C:=e^{6\kk}+\qtr$, and, $r\le\half r_\infty$ (because $\bdi A\cap\bOm_\Pi\ne\emptyset$).  If $r\ge\half r_\infty e^{-m/2}$, then
\begin{gather*}
  \half r_\infty \le r e^{m/2}=d e^{-m/2} \le C r_\infty e^{-m/2}
  \intertext{whence}
  \bp(a)=m\le 12\kk+4\,.
  \intertext{Assume $r\le\half r_\infty e^{-m/2}$.  We claim that $R':=(\half r_\infty)^2/r\in[d,\frac49 R]$.  Therefore}
  A':=\Set{z\in\mfC|r<|z-c|<R'}\csubann A \quad\text{is a concentric subannulus of $A$}
  \intertext{with $\Sone(A')=\Sone(c;\half r_\infty)$.  From \eqref{E:3:a4UPS}, $\Sone(A')\cap\Del=\emptyset$, so by earlier work,}
  \bp(a)=\half\md(A)=\log\frac{d}{r}\le\log\frac{R'}{r}=\md(A')\le 2M+\log4\,.
  \intertext{To check the claim,  note that}
  d=r e^m = \frac{(re^{m/2})^2}r \le \frac{(\half r_\infty)^2}r = R' \le \frac{(\frac23 d)^2}r =\frac{4rR}{9r} = \frac49R\,.
\end{gather*}

\medskip

Having established that $\bp_\Pi\le M':= \max\{2M+\log4,12\kk+4\}$ in $\Om_{\tilde\Del}$, we now know that $\lam_\Pi\,ds$ and $\del_\Pi^{-1}\,ds$ are $K$-\BL\ equivalent in $\Om_{\tilde\Del}$ with $K:=K(M)=\kk+M'$.  As explained above, any hyperbolic or quasihyperbolic geodesic in $\Om_\Pi$ with endpoints in $\Om_\Del$ lies in $\Om_{\tilde\Del}$.  Therefore the identity map $(\Om_\Pi,h_\Pi)\to(\Om_\Pi,k_\Pi)$ is $K$-\BL.
\qed

\subsection{Proof of Theorem~\ref{TT:example}}  \label{s:example} %
Here we concoct a plane domain $\Om$ and prove that the assertions of \rf{TT:example} hold for $\Om$.  We require the following technical fact about \qsc\ maps; this must be folklore, but we do not know a reference.

\begin{lma} \label{L:QSid} %
Let $(a_n)_0^\infty$ be a strictly increasing sequence in $\mfR\sm\{\infty\}$ with $a_{n+1}/a_n\to+\infty$ as $n\to+\infty$.  Put $A:=\Set{a_n|n\ge0}\cup\{\infty\}\subset\Rh$.  Then every QS \homeo\ $f:A\to A$ is ``eventually the identity''; i.e., there is an $N$ such that for all $n\ge N$, $f(a_n)=a_n$.\footnote{%
Here we use the chordal distance from $\Rh$.}
\end{lma} 
\begin{proof}%
To start, note that as $f$ is a \homeo, it is a bijection and $f(\infty)=\infty$ (as $\infty$ is the only non-isolated point of $A$).

We exhibit $p,q$ such that for all $n\ge1$, $f(a_{p+n})=a_{q+n}$.  Then $f$ maps $\{a_0,a_1,\dots,a_p\}$ bijectively onto $\{a_0,a_1,\dots,a_q\}$.  Hence $p=q$ and our claim is established.

We need only produce a $p$ such that for all $n\ge p$, $f(a_n)<f(a_{n+1})$ with $f(a_n)$ and $f(a_{n+1})$ \emph{adjacent} (meaning that if $f(a_n)=a_m$, then $f(a_{n+1})=a_{m+1}$).  Indeed, given such a $p$ we simply let $q$ be the unique integer with $a_q=f(a_p)$.

Below we use the \qsy\ of $f$ to verify that there is an $N$ such that for all $n\ge N$, $f(a_n)<f(a_{n+1})$.  Let $M$ be the unique integer with $a_M=f(a_N)$.  Thus for all $n\ge N$, $a_M\le f(a_n) < f(a_{n+1})$.

We claim that there are a finite number of $m\ge M$ such that $a_m\notin f\bigl(\Set{a_n|n\ge N}\bigr)$.  Indeed, if $m$ is such, then $a_m\in f\bigl(\{a_0,a_1,\dots,a_{N-1}\}\bigr)$, so, there are at most $N$ such $m$.  Let $\ell$ be the \emph{largest} of all these $m$.

Thus there is a $p>N$ such that $a_M\le f(a_{p-1}) < a_\ell < f(a_{p})$, but for all $n\ge p$, $f(a_n)$ and $f(a_{n+1})$ are adjacent.  This is the sough after $p$.

\medskip

To produce $N$, assume $f$ is $\eta$-QS.  Pick $\tau\in(0,1)$ so that $t\in(0,\tau)\implies\eta(t)<\half$.  Since
\begin{gather*}
  \frac{\chi(a_{n+1},\infty)}{\chi(a_{n},\infty)} \le 2
    \frac{1+a_{n}}{1+a_{n+1}}\to0\,,
  \intertext{there exists an $N$ such that for all $n\ge N$,}
  t:=\frac{\chi(a_{n+1},\infty)}{\chi(a_{n},\infty)}<\tau\,.
  \intertext{Then for such $n$,}
  \frac{\chi(fa_{n+1},f\infty)}{\chi(fa_{n},f\infty)} \le \eta(t) < \half
  \intertext{whence}
  \frac4{1+f(a_{n+1})} \le 2\chi(fa_{n+1},\infty) \le \chi(fa_n,\infty)
    \le \frac4{1+f(a_n)}
\end{gather*}
and therefore $f(a_n)<f(a_{n+1})$ as asserted.
\end{proof}%

\begin{pf}{Proof of \rf{TT:example}} %
Let $(a_n)_1^\infty$ be a strictly decreasing sequence in $(-\infty,0)$ with $a_{n+1}/a_n\to+\infty$ as $n\to+\infty$.  Put $a_0:=0$, $A:=\Set{a_n|n\ge0}$, and $\Om:=\mfC\sm A$.  We verify that the assertions of \rf{TT:example} hold for $\Om$.\footnote{%
The diligent reader will check that there is no choice of $r_p>0$ such that the hypotheses in \rf{ss:NKA} hold for $\Pi=A$ with $\Om=\mfC$; the uniform perfectness of $\RS\sm\Om_\Del$ always fails.}

It is straightforward to check that $\Om$ is a uniform domain; \cite[Lemma~8.4]{DAH-thesis} provides a convenient criterion here.

We employ \cite[Lemma~3.14]{BHK-unif} and its hyperbolic counterpart \rf{P:For unif=>QS}(f,g).  To this end, note that the point $o:=1\in\Om$ has the property that $\sig(1)=\dist_\sig(1,\hat\bd\Om)=\max_\Om\sig=\pi/2$.  Also, $[1,+\infty)\subset\mfR\subset\RS$ is both a hyperbolic and a quasihyperbolic geodesic ray in $\Om$ from $1$ to the boundary point at infinity.

Let $h_\veps=h_{\veps,o}$ and $k_\veps=k_{\veps,o}$ denote the standard visual distances on $\bd_G(\Om,h)$ and $\bd_G(\Om,k)$ repectively; as in \rf{ss:GH&U}, the visual parameter $\veps\in(0,\veps_0]$ and $o=1\in\Om$ is the fixed base point.

Thanks to \cite[Theorem~3.6]{BHK-unif} and \rf{T:unif=>QS}, we know that both $\bd_G(\Om,k)$ and $\bd_G(\Om,h)$ are naturally equivalent to $\hat\bd\Om=\bOm\cup\{\infty\}=A\cup\{\infty\}$.  We use $a_n$ to denote any of: $a_n\in A\subset\bOm$, the corresponding point in $\bd_G(\Om,k)$, or  the corresponding point in $\bd_G(\Om,h)$.  We also let $\zeta$ denote the boundary point at infinity in $\hat\bd\Om$ and the corresponding points in $\bd_G(\Om,k)$ and $\bd_G(\Om,h)$.

Given a boundary point $\xi=a_n$, we easily check that the associated point $x\in[o,\zeta)=[1,\infty)$ given by \cite[Lemma~3.14]{BHK-unif} or \rf{P:For unif=>QS}(f) is $x=x_n:=-a_n$.  Employing \rf{P:For unif=>QS}(g) (or its quasihyperbolic counterpart in \cite{BHK-unif}) we deduce that for each distance function $d\in\{h,k\}$,
\begin{gather*}
  C^{-1} e^{-\veps d(x_n,1)} \le d_\veps(a_n,\zeta) \le
    C e^{-\veps d(x_n,1)}\,,
  \intertext{where $C=C(\veps)$.  It follows that}
  \frac{k_\veps(a_{n+1},\zeta)}{k_\veps(a_n,\zeta)} \le
    C^2 e^{-\veps k(x_n,x_{n+1})}
  \intertext{and that}
  \frac{h_\veps(a_{n+1},\zeta)}{h_\veps(a_n,\zeta)} \ge
    C^{-2} e^{-\veps h(x_n,x_{n+1})}\,.
\end{gather*}

Note that
\begin{gather*}
  \forall\; n\ge1\,,\quad  k(x_n,x_{n+1})=\log\frac{x_{n+1}}{x_n}
  \intertext{and, arguing as in \rf{X:h est},}
  \forall\; n\ge N\,, \quad  h(x_n,x_{n+1})\le
    4+\pi\log\Bigl(\half\log\frac{x_{n+1}}{x_n}\Bigr)
\end{gather*}
where $N$ is chose so that $n\ge N\implies \log(x_{n+1}/x_n)>2$.

Now suppose there were a power \qsy\ $\bd_G(\Om,k)\xra{f}\bd_G(\Om,h)$; say, $f$ is $\eta$-QS with $\eta(t):=H\bigl( t^\alf \vee t^{1/\alf} \bigr)$ for some constants $H>0$ and $\alf\in(0,1]$.  Then $f$ would induce a \qsy\ of $A\cup\{\zeta\}$ onto itself, and hence by \rf{L:QSid} we would know that $f(\zeta)=\zeta$ and $f(a_n)=a_n$ for all $n\ge N$.

It would then follow that: for all $n\ge N$,
\begin{gather*}
  C^{-2} e^{-\veps h(x_n,x_{n+1})}
    \le \frac{h_\veps(a_{n+1},\zeta)}{h_\veps(a_n,\zeta)}
    \le\eta\biggl(\frac{k_\veps(a_{n+1},\zeta)}{k_\veps(a_n,\zeta)}\biggr)
    \le H C^{2\alf} e^{-\alf\veps k(x_n,x_{n+1})}
  \intertext{or,}
  e^{\veps\bigl(\alf k(x_n,x_{n+1}) - h(x_n,x_{n+1})\bigr)}
    \le H C^{1+2\alf}\,,
  \intertext{which in turn would imply that $\alf k(x_n,x_{n+1}) - h(x_n,x_{n+1})$ is bounded as $n\to+\infty$.  However, $L_n:=\half\log(x_{n+1}/x_n)\to+\infty$ as $n\to+\infty$, and,}
  \alf k(x_n,x_{n+1}) - h(x_n,x_{n+1} \ge 2\alf L_n - 4 - \pi \log L_n
\end{gather*}
from which we deduce that $\alf k(x_n,x_{n+1}) - h(x_n,x_{n+1})$ is \emph{not} bounded as $n\to+\infty$.  This contradiction means there cannot exist such a power \qsy.
\end{pf}
\bibliographystyle{amsalpha} 
\bibliography{mrabbrev,bib}  

\end{document} 

paper outline  %
Intro          %
Prelims        %
  basic info   %
  std notation %
Examples       %
Theorems       %
Questions      %


\section{Introduction}  \label{S:Intro} 
\subsection{bla}  \label{s:bla} 
\begin{Thm}  \label{TT:what} 
\end{Thm}                    

\begin{Cor}  \label{CC:..} 
\end{Cor}                  

\begin{thm} \label{T:ttt} %
\end{thm} 
\begin{proof}%
\end{proof}%

\begin{prop} \label{P:ppp} %
\end{prop} 
\begin{proof}%
\end{proof}%

\begin{lma} \label{L:lma} %
\end{lma} 

\begin{fact}  \label{F:ff} %
\end{fact} %
also use for:  remarks(s), examples, defns

\begin{rmks}  \label{R:} 
\hfill\break
blah blah blah
\smallskip\noindent(a)
\smallskip\noindent(a)

OR can use below but it gets seriously indented :-((
\begin{itemize}
  \item[ ]
  \item[(a)]
  \item[(b)]
\end{itemize}
\end{rmks} %

\begin{qstn} \label{Q:...} %

figures, pictures, diagrams etc %







This document is organized as follows: Section~\ref{S:Prelims} contains preliminary information including basic definitions and terminology as well as elementary and/or well-known facts.  We exhibit examples and examine ...  In Section~\ref{S:misc} we verify ....